\documentclass[12pt]{amsart}

\usepackage{amssymb}
\usepackage{enumerate}
\usepackage{subfigure}
\usepackage{setspace}
\usepackage{graphicx}

\numberwithin{equation}{section}

\newtheorem{thm}{Theorem}[section]
\newtheorem{prop}[thm]{Proposition}
\newtheorem{cor}[thm]{Corollary}
\newtheorem{lem}[thm]{Lemma}

\newtheorem{question}[thm]{Question}
\newtheorem{prob}[thm]{Problem}
\theoremstyle{definition}
\newtheorem{defn}{Definition}[section]
\newtheorem{ex}[defn]{Example}
\newtheorem{rem}[defn]{Remark}

\newcommand{\csn}{\mathbb{C}[S_n]}
\newcommand{\hnq}{H_n(q)}

\newcommand{\imm}[1]{\mathrm{Imm}_{#1}}

\newcommand{\ssm}{\smallsetminus}

\newcounter{countcases}

\newcommand{\PBS}[1]{\let\temp=\\#1\let\\=\temp}

\numberwithin{figure}{section}

\begin{document}
\author{Brendon Rhoades}
\title[Cyclic Sieving, Promotion, and Representation Theory]
{Cyclic Sieving, Promotion, and Representation Theory}

\bibliographystyle{../dart}

\date{\today}

\begin{abstract}
We prove a collection of conjectures of D. White \cite{WComm}, as well as some
related conjectures of Abuzzahab-Korson-Li-Meyer \cite{AKLM} and of Reiner and
White \cite{ReinerComm}, \cite{WComm}, regarding the cyclic sieving phenomenon
of Reiner, Stanton, and White \cite{RSWCSP} as it applies to jeu-de-taquin
promotion on rectangular tableaux.  To do this, we use Kazhdan-Lusztig theory
and a characterization of the dual canonical basis of $\mathbb{C}[x_{11},
\dots, x_{nn}]$ due to Skandera \cite{SkanNNDCB}.  Afterwards, we extend our
results to analyzing the fixed points of a dihedral action on rectangular
tableaux generated by promotion and evacuation, suggesting a possible sieving
phenomenon for dihedral groups.  Finally, we give  applications of this theory
to cyclic sieving phenomena involving reduced words for the long elements of
hyperoctohedral groups and noncrossing partitions.

\end{abstract}

\maketitle

\section{Introduction}

Suppose that we are given a finite set $X$ equipped with the
action of a finite cyclic group $C$ generated by $c$.  In studying the
combinatorial structure of the action of $C$ on $X$, it is natural to ask for
the sizes of the fixed point sets $X^1 = X, X^c, X^{c^2}, \dots,
X^{c^{(|C|-1)}}$.  Indeed, the cardinalities of the above sets determine the
cycle structure of the image of $c$ under the canonical homomorphism $C
\rightarrow S_X$, so from a purely enumerative standpoint these fixed point set
sizes determine the action of $C$ on $X$.

Reiner, Stanton, and White studied such actions and associated polynomials to
them which encode the sizes of all of the above fixed point sets at the same
time \cite{RSWCSP}.  Following their 2004 paper, we make the following
definition.

\begin{defn}
Let $C$ be a finite cyclic group acting on a finite set $X$ and let $c$ be a
generator of $C$.  Let $\zeta \in \mathbb{C}$ be a root of unity having the
same multiplicative order as $c$ and let $X(q) \in \mathbb{Q}[q]$ be a
polynomial.  We say that the triple $(X, C, X(q))$ \emph{exhibits the cyclic
sieving phenomenon (CSP)} if for any integer $d \geq 0$ we have that the fixed
point set cardinality $|X^{c^d}|$ is equal to the polynomial evaluation
$X(\zeta^d)$.
\end{defn}

This definition generalizes Stembridge's notion of the $q = -1$ phenomenon
\cite{StemMinus}, \cite{StemTab} which restricts the above definition to the
case where the cyclic group $C$ has order $2$ or, equivalently, when we are
given a set $X$ equipped with an involution $X \rightarrow X$.

A few remarks are in order.  First, notice that since the identity element of
$C$ fixes every element of $X$, we have that $X(1) = |X|$ whenever $(X, C,
X(q))$ exhibits the CSP.  Also, it is easy to show that given a finite cyclic
group $C$ acting on a finite set $X$, a polynomial $X(q)$ such that $(X, C,
X(q))$ exhibits the CSP is unique modulo the ideal in $\mathbb{Q}[q]$ generated
by the cyclotomic polynomial $\Phi_{|C|}(q)$.  Finally, given any finite cyclic
group $C$ acting on a finite set $X$, it is possible to show that the triple
$(X, C, X(q))$ exhibits the CSP, where $X(q) = \sum_{i=0}^{|C|-1} a_i q^i$ and
$a_i$ is the number of $C$-orbits in $X$ with stabilizer order dividing $i$.
Typically the interest in a CSP is that $X(q)$ may be taken to be a natural
polynomial deformation of a formula enumerating $|X|$.

These polynomial deformations often come from the theory of $q$-numbers.  For
any $n \in \mathbb{N}$, define the $q$-analogue $[n]_q := \frac{q^n - 1}{q - 1}
= 1 + q + q^2 + \cdots + q^{n-1}$.  Following the case of natural numbers, we
define further $[n]!_q := [n]_q [n-1]_q \dots [1]_q$ and
${n \brack k}_q := \frac{[n]!_q}{[k]!_q [n-k]!_q}$.  It is well known that ${n
\brack k}_q$ is the generating function for partitions which fit inside an
$(n-k)$ by $k$ rectangle and is, therefore, a polynomial \cite{StanEC1}.  We
first give the `canonical' example of the CSP.

\begin{thm} (Reiner-Stanton-White 2004 \cite{RSWCSP})
Fix two positive integers $k \leq n$.  Let $X$ be the set of all subsets of
$[n]$ having size $k$ and let $C = \mathbb{Z}/n\mathbb{Z}$ act on $X$ via the
long cycle $(1, 2, \dots, n) \in S_n$.

Then, the triple $(X, C, X(q))$ exhibits the CSP, where $X(q) = {n \brack
k}_q$.
\end{thm}

Reiner et. al. also proved a version of the above result for the case of
multisets.

\begin{thm} (Reiner-Stanton-White 2004 \cite{RSWCSP})
Fix two positive integers $k$ and $n$.  Let $X$ be the set of all $k$-element
multisets of $[n]$.  Let $C = \mathbb{Z}/n\mathbb{Z}$ act on $X$ by the long
cycle $(1, 2, \dots, n) \in S_n$.

Then, the triple $(X, C, X(q))$ exhibits the CSP, where $X(q) = {n+k-1 \brack
k}_q$.
\end{thm}

In each of these results and in all of the rest of the CSPs appearing in this
paper, the set $X$ is a set of combinatorial objects and $C$ is generated by a
natural combinatorial operator on $X$.  We shall see that Theorems 1.1 and 1.2
are both implied by CSPs regarding certain sets of tableaux and an action given
by the sliding algorithm of jeu-de-taquin promotion.  Postponing  definitions
until Section 2, we state our three main results.  The following was conjetured
by D. White \cite{WComm}.

\begin{thm}
Let $\lambda \vdash n$ be a rectangular partition and let $X = SYT(\lambda)$.
Let $C = \mathbb{Z} / n \mathbb{Z}$ act on $X$ by jeu-de-taquin promotion.

Then, the triple $(X, C, X(q))$ exhibits the cyclic sieving phenomenon, where
$X(q)$ is the $q-$analogue of the hook length formula
\begin{equation*}
X(q) = f^{\lambda}(q) := \frac{[n]!_q}{\Pi_{(i,j) \in \lambda}[h_{ij}]_q}.
\end{equation*}
\end{thm}

There are several interpretations of the polynomial $X(q)$ in the above result.
In addition to being the natural $q$-analogue of the hook length formula, up to
a power of $q$ we may also interpret $X(q)$ to be the fake degree polynomial
corresponding to the $\lambda$-isotypic component of the action of $S_n$ on the
coinvariant algebra
$\mathbb{C}[x_1, \dots, x_n] / \mathbb{C}[x_1, \dots, x_n]_+^{S_n}$ (see
Proposition 4.1 Stanley \cite{StanInv} together with Corollary 7.21.5 of
\cite{StanEC2}).
Moreover, we have that $X(q)$ is equal to the Kostka-Foulkes polynomial
$K_{\lambda,1^n}(q)$ corresponding to $\lambda$ and the composition $(1^n)$ of
length $n$ consisting entirely of $1$'s.  Finally, up to a power of $q$, $X(q)$
is equal to the $q-$analogue of the Weyl dimension formula for the
$(1,1,\dots,1)$-weight space of the irreducible representation of
$GL_n(\mathbb{C})$ having highest weight $\lambda$.
There is a column strict version of the above result, also conjectured by White
\cite{WComm}.

\begin{thm}
Let $k \geq 0$ and let $\lambda \vdash n$ be a rectangular partition.  Let $X =
CST(\lambda, k)$ and let $C = \mathbb{Z}/k\mathbb{Z}$ act on $X$ via
jeu-de-taquin promotion.

Then, the triple $(X, C, X(q))$ exhibits the cyclic sieving phenomenon, where
$X(q)$ is a $q-$shift of the principal specialization of the Schur function
\begin{equation*}
X(q) := q^{-\kappa(\lambda)} s_{\lambda}(1,q,q^2,\dots,q^{k-1}),
\end{equation*}
where $\kappa$ is the statistic on partitions $\lambda = (\lambda_1, \lambda_2,
\dots )$ given by
\begin{equation*}
\kappa(\lambda) = 0 \lambda_1 + 1 \lambda_2 + 2 \lambda_3 + \dots .
\end{equation*}
\end{thm}

There are several interpretations of the polynomial $X(q)$ in the above
theorem, as well.  Just as standard tableaux of a fixed shape are enumerated by
the hook length formula, the hook content formula enumerates column strict
tableaux of a fixed shape (with uniformly bounded entries).  Up to a power of
$q$, the polynomial $X(q)$ is the $q-$analogue of the hook content formula.
Also, $X(q)$ is equal to MacMahon's generating function for plane partitions
which fit inside a box having dimensions $\lambda_1$ by $\lambda'_1$ by $k -
\lambda_1$ weighted by number of boxes \cite{McMComb}.  This latter
interpretation can be easily seen via the obvious bijection between column
strict tableaux and plane partitions.  Finally, up to a power of $q$, $X(q)$ is
equal to the $q-$analogue of the Weyl dimension formula corresponding to the
irreducible representation of $GL_k(\mathbb{C})$ having highest weight
$\lambda$.

The previous two results have concerned standard tableaux of fixed rectangular
shape and column strict tableaux of fixed rectangular shape and arbitrary
content, respectively.  We can also formulate a result suggested by Reiner and
White \cite{ReinerComm} \cite{WComm} concerning column strict tableaux of fixed
shape and fixed content.  Specifically, suppose that $\lambda \vdash n$ is
rectangular and $\alpha \models n$ is a composition with length $\ell(\alpha) =
k$.   Also suppose that $\alpha$ has some cyclic symmetry, i.e., there exists
some $d | k$ such that the $d^{th}$ cyclic shift operator preserves $\alpha$.
It can be shown that the action of promotion on column strict tableaux acts
cyclically on content compositions.  Therefore, the $d^{th}$ power of promotion
acts on the set of column strict tableaux of shape $\lambda$ and content
$\alpha$.

\begin{thm}
Let $\lambda \vdash n$ and $\alpha \models n$ be as above.  Let $C = \mathbb{Z}
/ (\frac{k}{d} \mathbb{Z}) = \langle c \rangle$ act on the set of column strict
tableaux of shape $\lambda$ and content $\alpha$ by the $d^{th}$ power of
promotion.  Let $\zeta$ be a primitive $\frac{k}{d}^{th}$ root of unity.

Then, for any $m \geq 0$, the number of fixed points under the action of the
$m^{th}$ power of promotion is equal to the modulus $|K_{\lambda,
\alpha}(\zeta^m)|$, where $K_{\lambda, \alpha}(q)$ is the Kostka-Foulkes
polynomial.
\end{thm}

This result is almost, but not quite, a CSP.  Specifically, since the
evaluation of a Kostka-Foulkes polynomial at a root of unity may be strictly
negative, we do not have that the action of the $d^{th}$ power of promotion on
the relevant set of tableaux together with $K_{\lambda,\alpha}(q)$ exhibits the
CSP.  However, we do have cyclic sieving `up to modulus'.  Moreover, since the
$q$-hook length formula is (up to a power of $q$) a special case of a
Kostka-Foulkes polynomial, it follows that (again up to modulus) our result on
standard tableaux is a special case of this latter result.  It will turn out
that Theorem 1.5 is a weight space refinement of Theorem 1.4.

The bulk of the remainder of this paper is devoted to the proofs of the above
Theorems 1.3, 1.4, and 1.5.  We now give our overarching philosophy which was
also used by Stembridge \cite{StemTab} in relation to the action of evacuation
on tableaux and the $q = -1$ phenomenon.  Suppose that we are given some finite
set $X$ equipped with the action a finite cyclic group $C = \langle c \rangle$
and a polynomial $X(q)$ and we wish to show that $(X, C, X(q))$ exhibits the
CSP.  On its face, this is a purely enumerative problem - if we could find some
formula for the evaluation of $X(q)$ at appropriate roots of unity and equate
this with some formula enumerating the fixed point set of $X$ under the action
of appropriate elements of $C$, we would be done.  This direct approach, when
possible, has its merits.  CSPs are certainly of enumerative interest and in
many cases show that previously studied natural $q$ analogues of counting
formulas for fundamental combinatorial sets encode information about fixed
point sets under fundamental combinatorial operators.

However, this na\"ive approach can often be very difficult.  For example, while
direct combinatorial proofs of the CSPs for subsets and multisets in
\cite{RSWCSP} exist, there is no known enumerative proof of the CSPs for
rectangular tableaux given above.  We instead take the viewpoint that, instead
of being purely enumerative, CSPs often conceal deeper algebraic structure
which can be exploited in their proofs.

Specifically, suppose that we have some complex vector space $V$ with basis $\{
f_x \,|\, x \in X \}$ indexed by the elements of $X$.  Suppose further that $V$
is equipped with the action of a group $G$ and denote the corresponding
representation by $\rho: G \rightarrow GL(V)$.  Finally suppose that we can
find some element $g \in G$ whose action on $V$ is given by the formula
\begin{equation}
\rho(g)(f_x) =  f_{c \cdot x},
\end{equation}
for all $x \in X$.  That is, the matrix for the action of $g$ with respect to
the given basis of $V$ is the permutation matrix corresponding to the action of
$c$ on $X$.  It is immediate that for any $d \geq 0$, the number of fixed
points $|X^{c^d}|$ of the action of $c^d$ on $X$ is equal to the trace of the
linear operator $\rho(g^d)$.  In symbols, if $\chi: G \rightarrow \mathbb{C}$
is the character of our representation,
\begin{equation}
| X^{c^d}| = \chi(g^d),
\end{equation}
for all $d \geq 0$.  So, we have reduced our problem of enumerating fixed point
sets to the evaluation of a certain character at a certain group element.  If
we can interpret this character evaluation as an appropriate root of unity
evaluation of $X(q)$, then our CSP will be proved.

At first glance, this approach may seem to complicate matters.  We must first
model our action of $C$ on $X$ in a representation theoretic context.
Moreover, the evaluation of the character $\chi(g^d)$ may be no easier than the
enumeration of the fixed point set $X^{c^d}$, particularly if the action of $C$
on $X$ is easy to understand.  However, character theory is a well studied
subject and provides us with much artillery with which to attack the former
problem.  An elementary and fundamentally important example of this is that if
$h \in G$ is a group element which is conjugate to $g$, we have that $\chi(g^d)
= \chi(h^d)$, where the latter character may be easier to compute.  So, under
nice conditions, the problem of proving a CSP may be reduced to a problem of
studying conjugacy in a group.  And, as a bonus, this method of proving a CSP
reveals representation theoretic structure that may have previously gone
unnoticed.  In our situation, we will gain some understanding of why the
hypothesis that our partitions be rectangular is necessary in our results.  It
is this approach that we will use in proving our CSPs.

These representation theoretic methods, however, come with some fine print.  It
is often too much to ask that an equation of the form (1.1) hold.  Sometimes we
must content ourselves with finding a group element $g \in G$ such that for all
$x \in X$ we have
\begin{equation}
\rho(g)(f_x) = \gamma(x) f_{c \cdot x},
\end{equation}
where $\gamma: X \rightarrow \mathbb{C}$ is some function, hopefully as simple
as possible.  Moreover, the identification of a polynomial evaluation at a root
of unity with a character evaluation may require more involved techniques than
elementary group conjugacy.  In particular, this part of the proof of our
result on standard tableaux will use results from Springer's theory of regular
elements.

The remainder of this paper is organized as follows.  In Section 2 we review
some tableaux theoretic definitions, the algorithm of jeu-de-taquin promotion,
and some representation theoretic tools we will be using in the proofs of our
CSPs (specifically Kazhdan-Lusztig theory).  In Section 3 we prove our CSP for
standard tableaux.  In Section 4 we derive a slightly new perspective on the
irreducible representations of the general linear group which will lead to a
proof of our CSP for column strict tableaux in Section 5.  
In Section 6 we extend the
general philosophy of the proofs in Section 5 to get our result concerning
column strict tableaux of fixed content.  In Section 7 we extend our results on
cyclic actions and prove some results which enumerate fixed points under
combinatorial \emph{dihedral} actions.  In Section 8 we derive corollaries of
our fixed point results on tableaux for other combinatorial actions involving
handshake patterns and the reflection group $B_n$.
We close in Section 9 with some open questions.

\section{Tableaux and Representation Theory Background}

We begin this section by gonig over the definitions of standard and column
strict tableaux, as well as the definition of the action of jeu-de-taquin
promotion.  For a more leisurely introduction to this material, see \cite{Sag}
or \cite{StanEC2}.

Given a positive integer $n$, a \emph{partition} $\lambda$ of $n$ is a weakly
decreasing sequence of positive integers $\lambda = (\lambda_1 \geq \lambda_2
\geq \dots \geq \lambda_k)$ such that $\lambda_1 + \lambda_2 + \dots +
\lambda_k = n$.  The number $k$ is the \emph{length} $\ell(\lambda)$ of
$\lambda$.  The number $n$ is the \emph{size} $|\lambda|$ of $\lambda$.   We
write $\lambda \vdash n$ to indicate that $\lambda$ is a partition of $n$.  We
will sometimes use exponential notation to write repeated parts of partitions,
so that $4^2 3^3 1$ is the partition
$(4,4,3,3,3,1) \vdash 18$.

We identify partitions $\lambda \vdash n$ with their \emph{Ferrers diagrams},
i.e., $\lambda$ is identified with the subset of the lower right quadrant of
$\mathbb{Z} \times \mathbb{Z}$ given by $\{ (i,-j) \,|\, i \in [\lambda_j] \}$.
For example, the Ferrers diagram of the partition $(4, 4, 3, 1) \vdash 12$ is
given by
\begin{equation*}
\begin{array}{ccccc}
\bullet  & \bullet  & \bullet  & \bullet  &\\
\bullet  & \bullet  & \bullet  & \bullet  &\\
\bullet  & \bullet  & \bullet  &       &\\
\bullet  &          &             &       &.
\end{array}
\end{equation*}
Given a coordinate $(i,-j)$ in the Ferrers diagram of $\lambda$, the \emph{hook
length} $h_{ij}$ at $(i,-j)$ is the number of dots directly south or directly
east of $(i,-j)$, the dot $(i,j)$ included.  In the above partition $h_{12} =
6$.
A partition is said to be \emph{rectangular} if its Ferrers diagram is a
rectangle.  The \emph{conjugate} $\lambda'$ of a partition $\lambda$ is the
partition whose Ferrers diagram is obtained by reflecting the Ferrers diagram
for $\lambda$ across the line $y = -x$.  So, $(4,4,3,1)' = (4,3,3,2)$.

Given partitions $\lambda$ and $\mu$ such that we have a set theoretic
containment of Ferrers diagrams $\mu \subseteq \lambda$, we define the
\emph{skew partition} $\lambda / \mu$ to be the set theoretic difference
$\lambda \ssm \mu$ of Ferrers diagrams.  The \emph{size} of $\lambda / \mu$ is
the difference $|\lambda| - |\mu|$ and we write $\lambda / \mu \vdash n$ to
denote that the skew partition $\lambda / \mu$ has size $n$.

Given a positive integer $n$, a \emph{composition} $\alpha$ of $n$ is a finite
sequence of nonnegative integers $\alpha = (\alpha_1, \alpha_2, \dots,
\alpha_k)$ which satisfies
$\alpha_1 + \alpha_2 + \cdots + \alpha_k = n$.  In particular, some of the
$\alpha_i$ may be zero.  We write $\alpha \models n$ to denote that $\alpha$ is
a composition of $n$.  The number $n$ is called the \emph{size} $|\alpha|$ of
$\alpha$ and the number $k$ is called the \emph{length} $\ell(\alpha)$ of
$\alpha$.
If $\alpha \models n$ and $\ell(\alpha) = k$, the composition $\alpha$ defines
a function $[n] \rightarrow [k]$ given by sending every number in the interval
$(\alpha_1 + \cdots + \alpha_{i-1},
\alpha_1 + \cdots + \alpha_i ]$ to the number $i$ for $i = 1, 2, \dots, k$.  We
denote this function by $\alpha$, as well.  For example, if $\alpha = (0, 2, 1,
0, 1)$, then $\alpha: [4] \rightarrow [5]$ is given by $\alpha(1) = 2,
\alpha(2) = 2, \alpha(3) = 3, \alpha(4) = 5$.
Given a partition $\alpha = (\alpha_1, \alpha_2, \dots, \alpha_k)$ of length
$k$, define $c_k \cdot \alpha$ to be the cyclically rotated composition
$(\alpha_2, \alpha_3, \dots, \alpha_k, \alpha_1)$.  This defines an action of
the order $k$ cyclic subgroup of $S_k$ generated by $(1,2,\dots,k)$ on the set
of all compositions of $n$ having length $k$.

Let $\lambda \vdash n$ and let $\alpha \models n$ be a composition with
$\ell(\alpha) = k$.  A \emph{column strict tableau} $T$ of shape $\lambda$ with
content $\alpha$  is a filling of the Ferrers diagram of $\lambda$ with
$\alpha_1$ $1's$, $\alpha_2$ $2's$, $\dots , \alpha_k$ $k's$ such that the
numbers increase strictly down every column and weakly across every row.  An
example of a column strict tableau of shape $(4, 4, 3, 1)$ is
\begin{equation*}
\begin{array}{ccccc}
1 & 1 & 3 & 4 & \\
3 & 3 & 4 & 6 & \\
4 & 5 & 5 &   & \\
6 &   &   &   &.
\end{array}
\end{equation*}
Given a partition $\lambda \vdash n$ and a composition $\alpha \models n$, the
set of all column strict tableaux of shape $\lambda$ and content $\alpha$ is
denoted $CST(\lambda, k, \alpha)$, where $k := \ell(\alpha)$ reminds us of the
maximum possible entry in our tableau.  The set of \emph{all} column strict
tableaux of shape $\lambda$ with entries at most $k$ is denoted $CST(\lambda,
k)$, so that
$CST(\lambda, k) = \bigcup_{\alpha \models n, \ell(\alpha) = k} CST(\lambda, k,
\alpha)$.  The tableaux obtained by requiring that entries increase weakly down
columns and strictly across rows are called \emph{row strict} and we have the
analogous definitions of $RST(\lambda, k, \alpha)$ and $RST(\lambda, k)$.
Finally, a tableau is called \emph{standard} if its content composition
consists entirely of $1's$.  We denote by $SYT(\lambda)$ the set of all
standard tableaux of shape $\lambda$ and make the canonical identification of
$SYT(\lambda)$ with $CST(\lambda, n, 1^n)$ for $\lambda \vdash n$.
Given any partition $\lambda \vdash n$, we have the \emph{column superstandard}
tableaux $CSS(\lambda)$ obtained by filling in the boxes of $\lambda$ within
each column, going from left to right.  For example,
\begin{equation*}
\begin{array}{ccccc}
             & 1 & 4 & 7 & \\
CSS((3,3,2))=& 2 & 5 & 8 & \\
             & 3 & 6 &   & .
\end{array}
\end{equation*}

For a partition $\lambda \vdash n$, $f^{\lambda} := |SYT(\lambda)|$ is famously
enumerated by the Frame-Robinson-Thrall hook length formula \cite{FRTHook}:
\begin{equation}
f^{\lambda} = \frac{n!}{\Pi_{(i,-j) \in \lambda} h_{ij}}.
\end{equation}
Also given a positive integer $k$, the \emph{Schur function} associated to the
partition $\lambda$ in $k$ variables is
the element of the polynomial ring $\mathbb{C}[x_1, \dots, x_k]$ given by
\begin{equation}
s_{\lambda}(x_1, \dots, x_k) = \sum_{\alpha \models n, \ell(\alpha) = k}
\sum_{T \in CST(\lambda, k, \alpha)}
x_1^{\alpha_1} \dots x_k^{\alpha_k}.
\end{equation}
That is, the Schur function associated to $\lambda$ is the generating function
for column strict tableaux of shape $\lambda$ weighted by their content
compositions.

The sets $X$ involved in the most important CSPs proved in this paper will be
standard tableaux with fixed rectangular shape, column strict tableaux with
fixed rectangular shape and uniformly bounded entries, and column strict
tableaux with fixed rectangular shape and specified content.  The cyclic
actions on each of these sets will be based on jeu-de-taquin promotion, a
combinatorial algorithm which we presently outline.  
The action of promotion has received recent attention from
Stanley \cite{StanProm}, who considers a related action
on linear extensions of finite posets, and from Bandlow, Schilling, and
Thiery \cite{BSTProm}, who derive crystal theoretic uniqueness results related to
the action of promotion on column strict tableaux.

Roughly speaking, promotion acts on tableaux by deleting all of the highest
possible entries, sliding the remaining entries out while preserving the column
strict condition, and then increasing all entries by one and filling holes with one so that the
resulting object is a column strict tableaux.  More formally,
suppose that we are given a partition $\lambda \vdash n$ and a positive integer
$k$.  We define the jeu-de-taquin promotion operator $j: CST(\lambda, k)
\rightarrow CST(\lambda, k)$ as follows.  Given a tableau $T$ in $CST(\lambda,
k)$, first replace every $k$ appearing in $T$ with a dot.  Suppose that there
is some dot in the figure so obtained that is not contained in a continuous
strip of dots in the northwest corner.  Then, choose the westernmost dot
contained in a connected component of dots not in the northwest corner of the
figure.  Say this dot has coordinates $(a,b)$, where $a > 0$ and $b < 0$.  By
our choice of dot, at least one of the positions $(a-1,b)$ or $(a,b+1)$ must be
filled with a number - i.e., the positions immediately west or north of our
dot.  If only one of these positions has a number, interchange that number and
the dot.  If both of these positions have a number, interchange the greater of
these two numbers and the dot (if the two numbers are equal, interchange the
number at the northern position $(a, b+1)$ and the dot).  This interchange
moves the dot one unit north or one unit west.  Continue interchanging the dot
with numbers in this fashion until the dot lies in a connected component of
dots in the northwest corner of the resulting figure.  If all of the dots in
the resulting figure are not in a connected component in the northwest corner,
choose the westernmost dot contained in a connected component of dots not in
the northwest corner of the figure and slide this new dot to the northwest
corner in the same way.  Iterate this process until all dots are contained in a
connected component in the northwest corner of the figure.  Now increase all
the numbers in the figure by $1$ and replace the dots with $1's$.  The
resulting figure is $j(T)$.

\begin{ex}
Let $T$ be the following element of $CST((4,4,3,1),6)$:
\begin{equation*}
\begin{array}{cccccc}
     & 1 & 1 & 3 & 4 & \\
     & 3 & 3 & 4 & 6 & \\
T =  & 4 & 5 & 5 &   & \\
     & 6 &   &   &   &.
\end{array}
\end{equation*}
We compute the image $j(T)$ of $T$ under jeu-de-taquin promotion.
\begin{equation*}
\begin{array}{ccccccccccccccc}
     & 1 & 1 & 3 & 4 &         & 1 & 1 & 3 & 4       &        & \bullet  & 1 &
3 & 4        \\
     & 3 & 3 & 4 & 6 &         & 3 & 3 & 4 & \bullet  &       & 1        & 3 &
4 & \bullet \\
T =  & 4 & 5 & 5 &   & \mapsto & 4 & 5 & 5 &        & \mapsto & 3        & 5 &
5 &          \\
     & 6 &   &   &   &         & \bullet  & & & &             & 4        &   &
&
\end{array}
\end{equation*}
\begin{equation*}
\begin{array}{ccccccccccc}
        & \bullet  & \bullet  & 1 & 3 &         & 1 & 1 & 2 & 4 &       \\
        & 1        & 3        & 4 & 4 &         & 2 & 4 & 5 & 5 &       \\
\mapsto & 3        & 5        & 5 &   & \mapsto & 4 & 6 & 6 &   & = j(T)\\
        & 4        &          &   &   &         & 5 &   &   &   &
\end{array}
\end{equation*}
Notice that the content of $T$ is $(2,0,3,3,2,2)$ and the content of $j(T)$ is
$(2,2,0,3,3,2)$ - that is, $j$ acts by cyclic rotation on content compositions.
Also notice that the result of applying $j$ to $T$ depended on our considering
$T$ as an element of $CST((4,4,3,1),6)$.  If we had considered $T$ as an
element of $CST((4,4,3,1),k)$ for any $k > 6$, then we would have
\begin{equation*}
\begin{array}{cccccc}
        & 2 & 2 & 4 & 5 &    \\
        & 4 & 4 & 5 & 7 &    \\
j(T) =  & 5 & 6 & 6 &   &    \\
        & 7 &   &   &   &.
\end{array}
\end{equation*}
\end{ex}

For the proof of the following lemma about promotion, see \cite{Sag} or
\cite{StanEC2}.

\begin{lem}
Let $T$ be in $CST(\lambda, k)$.\\
1.  The tableau $j(T)$ is well-defined, i.e., independent of the order in which
we chose dots to slide northwest in the algorithm.\\
2.  The tableau $j(T)$ is column strict with entries $\leq k$.
\\
3.  If $T$ has content $\alpha$, then $j(T)$ has content $c_k \cdot \alpha$.
\end{lem}

So, $j$ is indeed a well defined function $CST(\lambda, k) \rightarrow
CST(\lambda, k)$.  By running the defining algorithm for $j$ in reverse, we see
that $j$ is a bijection and we call its inverse $j^{-1}$ jeu-de-taquin
\emph{demotion}.  By Part 3 of the above lemma, $j$ restricts to an operator
$SYT(\lambda) \rightarrow SYT(\lambda)$ where we consider the `upper bound' $k$
to be equal to $|\lambda|$.  We call this restriction promotion, as well, and
retain the notation $j$ for it.

We now review the construction of the Kazhdan-Lusztig basis for the Hecke
algebra $\hnq$ and the symmetric group algebra $\csn$, the construction of the
Kazhdan-Lusztig cellular representations of $S_n$, and the interaction of these
representations with the combinatorial insertion algorithm of RSK.  See
\cite{BBCoxeter}, \cite{GarMcL}, \cite{Sag}, and \cite{StanEC2} for a more
leisurely introduction to this material.

A fundamental result in the representation theory of the symmetric group $S_n$
is that we have an isomorphism of $\csn$-modules
\begin{equation*}
\csn \cong \bigoplus_{\lambda \vdash n} f^{\lambda} W_{\lambda},
\end{equation*}
where $W_{\lambda}$ is the irreducible representation of $S_n$ indexed by the
partition $\lambda$.  Viewing the left regular representation as a well
understood object, it is natural to study bases of the left hand side of the
above isomorphism which allow the visualization of the decomposition on the
right hand side and, in particular, facilitate the study of the irreducibles
$W_{\lambda}$.
The `natural' basis $\{ w \,|\, w \in S_n \}$ fails miserably in this regard -
every element of the symmetric group acts as an $n!$ by $n!$ permutation matrix
with respect to this basis, rendering the above isomorphism invisible.
It turns out a basis for $\csn$ in which the above isomorphism is evident
arises in a natural way when one studies algebras which generalize
the symmetric group algebra.

In a fundamental 1978 paper \cite{KLRepCH}, Kazhdan and Lusztig studied the
representation theory of the Hecke algebra $\hnq$, which is a quantum
deformation of the symmetric group algebra $\csn$ and reduces to $\csn$ in the
specialization $q = 1$.  $\hnq$ admits an involution which restricts to the
identity on $\csn$ and it is natural to ask whether there are bases of $\hnq$
which are invariant under this involution. The answer is `yes', and up to
certain normalization conditions these bases are unique.  Specialization of
this basis of $\hnq$ at $q = 1$ yields a basis of $\csn$ which, as we will see,
can be viewed as `upper triangular' with respect to the above isomorphism.
Also amazingly, the avatars of the irreducibles $W_{\lambda}$ so obtained will
interact very nicely with combinatorial notions such as descent sets of
tableaux and RSK insertion.

The symmetric group $S_n$ has a Coxeter presentation with generators $s_1, s_2,
\dots, s_{n-1}$ and relations

\begin{align}
s_i s_j = s_j s_i & & \text{for $|i-j| > 1$,}\\
s_i s_j s_i = s_j s_i s_j
& & \text{for $|i-j|= 1$,}\\
s_i^2 = 1 &  & \text{for all $i$.}
\end{align}

We interpret $s_i$ to be the adjacent transposition $(i, i+1)$.  The
\emph{length} $\ell(w)$ of a permutation $w \in S_n$ is the minimum value of
$r$ so that $w = s_{i_1} \dots s_{i_r}$ for some adjacent transpositions
$s_{i_j}$.  Call such a minimal length word \emph{reduced}.  We define the
\emph{left descent set of w} to be the subset $D_L(w)$ of $[n-1]$ given by
$D_L(w) := \{ i \,|\, \ell(s_i w) < \ell(w) \}$.  The \emph{right descent set
of w}, $D_R(w)$, is defined analogously.  In this paper, we will denote
permutations in $S_n$ by their cycle decomposition, their expressions as words
in the Coxeter generators $s_i$, and by their one line notation.  In this
latter system, writing $w = w_1 w_2 \dots w_n$ means that $w$ sends $1$ to
$w_1$, $2$ to $w_2$, and so on.

Given a partition $\lambda = (\lambda_1, \lambda_2, \dots, \lambda_k) \vdash
n$, we define the \emph{Young subgroup of $S_n$ indexed by $\lambda$} to be the
subgroup $S_{\lambda}$ of $S_n$ which stabilizes the sets
\begin{equation*}
\{1, 2, \dots, \lambda_1 \}, \{\lambda_1 + 1, \dots, \lambda_1 + \lambda_2 \},
\dots, \{\lambda_1 + \cdots + \lambda_{k-1} + 1, \dots, \lambda_1 + \cdots +
\lambda_k = n\}.
\end{equation*}
\noindent For example, if $\lambda = (3,2,2) \vdash 7$, then the associated
Young subgroup of $S_7$ is given by
$S_{\lambda} = S_{\{1,2,3\}} \times S_{\{4,5\}} \times S_{\{6,7\}}$.
In general, we have a natural direct product decomposition $S_{\lambda} \cong
S_{\lambda_1} \times \cdots \times S_{\lambda_k}$ which implies the
corresponding order formula $|S_{\lambda}| = \lambda_1 ! \cdots \lambda_k !$.

$S_n$ comes equipped with a (strong) Bruhat order, the partial order given by
the transitive closure of $w \prec v$ if and only if there exists some (not
necessarily simple) reflection $t \in T := \bigcup_{w \in S_n, i \in [n-1]} w
s_i w^{-1}$ such that $v = tw$ and $\ell(w) \leq \ell(v)$.  Unless otherwise
indicated, writing $w \leq v$ for permutations $w$ and $v$ will always mean
comparability in Bruhat order.  The identity permutation $1$ is the unique
minimal element of $S_n$ under Bruhat order.  The \emph{long element} $w_o$
whose one-line notation is $n(n-1) \dots 1$ is the unique maximal element in
$S_n$ under Bruhat order.

Let $q$ be a formal indeterminate.  The \emph{Hecke algebra} $\hnq$ is the
$\mathbb{C}(q^{1/2})$-algebra with generators $T_{s_1}, T_{s_2}, \dots,
T_{s_{n-1}}$ subject to the relations:

\begin{align}
T_{s_i} T_{s_j} = T_{s_j} T_{s_i} & & \text{for $|i-j| > 1$,}\\
T_{s_i} T_{s_{j}} T_{s_i} = T_{s_{j}} T_{s_i} T_{s_{j}}
& & \text{for $|i-j|= 1$,}\\
T_{s_i}^2 = (1-q) T_{s_i} + q &  & \text{for all $i$.}
\end{align}

It turns out that if $w$ is a permutation in $S_n$ and $s_{i_1} \dots s_{i_r}$
is a reduced expression for $w$, then the Hecke algebra element
$T_w := T_{s_{i_1}} \dots T_{s_{i_r}}$ is independent of the choice of reduced
word for $w$.  Moreover, the set $\{ T_w \,|\, w \in S_n \}$ forms a basis for
$\hnq$ over the field $\mathbb{C}(q^{1/2})$.  Finally, it is obvious that the
specialization of the defining relations of $\hnq$ to $q = 1$ yields the
classical group algebra $\csn$.

It follows from the definition of $\hnq$ that the generator $T_{s_i}$ is
invertible for all $i$, with $T_{s_i}^{-1} = \frac{1}{q}(T_{s_i} - (1-q))$.
Therefore, for any permutation $w$ in $S_n$ we have that the algebra element
$T_w$ is invertible, being a product of invertible elements.  With this in
mind, define an involution $D$ of $\hnq$ by $D(q^{1/2}) = q^{-1/2}$ and $D(T_w)
= (T_{w^{-1}})^{-1}$ and extending linearly over $\mathbb{C}$.  In the latter
formula, the inverse in the subscript is taken in the symmetric group $S_n$ and
the inverse in the exponent is taken in the Hecke algebra $\hnq$.  Notice that
in the specialization at $q = 1$, the involution $D$ is just the identity map
on $\csn$.

\begin{thm}(Kazhdan-Lusztig \cite{KLRepCH})
There is a unique basis \\$\{ C'_w(q) = (q^{-1/2})^{\ell(w)} \sum_{v \in S_n}
P_{v,w}(q) T_v \,|\, w \in S_n \}$  of $\hnq$ such that\\
1. (Invariance) $D(C'_w(q)) = C'_w(q)$ for all permutations $w \in S_n$,\\
2. (Polynomality) $P_{w,v}(q) \in \mathbb{Z}[q]$ always,\\
3. (Normalization) $P_{w,w}(q) = 1$ for any $w \in S_n$,\\
4. (Bruhat compatibility) $P_{v,w} = 0$ unless $v \leq w$ in Bruhat order,
and\\
5. (Degree bound) The degree of $P_{v,w}(q)$ is at most $(1/2)(\ell(w) -
\ell(v) - 1)$.
\end{thm}

The basis in the above theorem is called the \emph{Kazhdan-Lusztig (KL) basis}
of $\hnq$.  Its specialization to $q = 1$ is a basis of $\csn$, but to make
some results in this paper look cleaner we will throw in some signs and call
$\{ C'_w(1) \,|\, w \in S_n \}$ the \emph{KL basis} of $\csn$, where $C'_w(1) =
\sum_{v \in S_n} (-1)^{\ell(w) - \ell(v)} P_{v,w}(1) v$.  These signs will not
seriously affect the representation theory.  The polynomials $P_{v,w}(q)$ are
the \emph{KL polynomials} and are notoriously difficult to compute for general
$v$ and $w$.  The KL basis of either $\hnq$ or $\csn$ leads to the definition
of the \emph{KL representation}, which is just the left regular representation
of either algebra viewed with respect to this basis.

In light of the degree bound in the above theorem, we recall a statistic
$\mu(v,w)$ on ordered pairs of permutations $v, w \in S_n$ by letting $\mu(v,w)
= [q^{(\ell(w) - \ell(v) - 1)/2}] P_{v,w}(q)$.  So, $\mu(v,w)$ is the
coefficient of the maximum possible power of $q$ in $P_{v,w}(q)$.  By Bruhat
compatibility and polynomality, we have that $\mu(v,w) = 0$ unless $v \leq w$
and also $\ell(v,w) := \ell(w) - \ell(v)$ is odd.  Moreover, we introduce a symmetrized
version of $\mu$ given by
$\mu[u,v] :=$ max$\{ \mu(u,v), \mu(v,u) \}$.

The  KL $\mu$ function can be used to get a recursive formula for the
$P_{u,v}(q)$ which will be of technical importance to us in what follows.
\begin{lem} (\cite{KLRepCH}, Equation 2.2c)
 If $u \leq w$ and $i \in D_L(w)$, then
\begin{equation}
   P_{u,w}(q) =
   q^{1-c} P_{s_i u, s_i w}(q) +
   q^c P_{u, s_i w}(q) -
   \sum_{s_i v < v} q^{\frac{\ell(v,w)}{2}}
   \mu(v, s_i w) P_{u,v}(q),
\end{equation}
where $c = 1$ if $i \in D_L(u)$ and $c = 0$ otherwise.
\end{lem}

The KL $\mu$ function is well behaved with respect to left and right
multiplication by the long element $w_o \in S_n$, as well as taking the
inverses of the permutations involved.  For a proof of the following lemma, see
for example \cite{BBCoxeter}.

\begin{lem}
Let $u, v \in S_n$.  We have that
$\mu(u, v) = \mu(w_o v, w_o u) =
 \mu(v w_o, u w_o) = \mu( w_o u w_o, w_o v w_o)$.  Also, we have that $\mu(u,v)
= \mu(u^{-1},v^{-1})$.
\end{lem}

Starting from the KL representation, we get a natural preorder on permutations
in $S_n$.  Specifically, for $u$ and $v$ in $S_n$, say that $u \leq_{L}^* v$ if
and only if there exists some $i$ in $[n-1]$ such that $C'_v(1)$ appears with
nonzero coefficient in the expansion of the product $s_i C'_u(1)$ in the KL
basis of $\csn$.  The transitive closure of $\leq_{L}^*$ defines a preorder
$\leq_L$ on $S_n$ (i.e., a reflexive, transitive relation which need not be
antisymmetric).  The preorder $\leq_L$ is called the \emph{left KL preorder}.
The \emph{right KL preorder} $\leq_R$ is defined in exactly the same way, but
by instead considering the \emph{right} regular representation of $\hnq$ or
$\csn$.  The \emph{two sided KL preorder} $\leq_{LR}$ is the transitive closure
of the union of $\leq_L$ and $\leq_R$.

Given any set $X$ equipped with a preorder $\leq$, we can define an equivalence
relation on $X$ by
$x \sim y$ if and only if $x \leq y$ and $y \leq x$.  The preorder $\leq$
induces a partial order on the equivalence classes $X / \sim$ also denoted
$\leq$ defined by $[x] \leq [y]$ if and only if for some elements $x \in [x]$
and $y \in [y]$, $x \leq y$.  The equivalence classes of permutations in $S_n$
so defined via the preorders $\leq_L, \leq_R,$ and $\leq_{LR}$ are called the
\emph{left, right, and two sided KL cells}, respectively.  Remarkably, these
cells can be identified via an explicit combinatorial algorithm.

It follows from counting the dimensions on both sides of the isomorphism $\csn
\cong \bigoplus_{\lambda \vdash n} f^{\lambda} W^{\lambda}$ that the sets of
permutations $w$ in $S_n$ and ordered pairs $(P,Q)$ of standard tableaux of the
same shape with $n$ boxes are in bijection.
The \emph{Robinson-Schensted-Knuth (RSK)} algorithm gives an explicit bijection
between these sets.
For a detailed definition of its algorithm see \cite{Sag} or \cite{StanEC2}.
In this paper, $w \mapsto (P(w), Q(w))$ will always mean that $w$ \emph{row}
inserts to $(P(w),Q(w))$.  For example, in $S_6$ we have that
\begin{equation*}
623415 \mapsto
\left(
\begin{array}{ccccccccc}
1 & 3 & 4 & 5 &   & 1 & 3 & 4 & 6 \\
2 &   &   &   & , & 2 &   &   &   \\
6 &   &   &   &   & 5 &   &   &
\end{array}
\right).
\end{equation*}
Define the \emph{shape} sh($w$) of a permutation $w$ to be the shape of either
$P(w)$ or $Q(w)$.  So,
sh(623415) = (4,1,1).

To further explore the interaction between the RSK algorithm and the algebraic
properties of the Coxeter group $S_n$, we introduce the notion of a descent set
of a tableau.
Given $\lambda \vdash n$ and $T \in SYT(\lambda)$, the \emph{descent set}
$D(T)$ of $T$ is the subset of $[n-1]$ defined by $i \in D(T)$ if and only if
$i+1$ occurs strictly south and weakRoughly speaking, promotion acts on tableaux by deleting all of the highest
possible entries, sliding the remaining entries out while preserving the column
strict condition, and then altering  entries and filling holes so that the
resulting object is a column strict tableaux.ly west of $i$ in $T$.  For example, if
\begin{equation*}
\begin{array}{ccccc}
    & 1 & 3 & 5 & \\
T = & 2 & 4 & 7 & \\
    & 6 &   &   &,
\end{array}
\end{equation*}
then $D(T) = \{1, 3, 5 \}$.  In the case of rectangular tableaux, we will later
generalize the descent set to another combinatorial object called an
\emph{extended descent set} which will help us greatly in proving our
representation theoretic results.

The RSK algorithm behaves in a predictable way with respect to taking inverses,
finding descent sets, and left and right multiplication by the long permutation
$w_o$.  To show this, we recall Sch\"utzenberger's combinatorial algorithm of
\emph{evacuation}.  Given an arbitrary partition $\lambda \vdash n$ and a
tableau $T$ in $CST(\lambda, k)$ for some $k \geq 0$, define the image $e(T)$
of $T$ under evacuation as follows.  First, embed $T$ in the northwest corner
of a very large rectangle.  Then rotate the rectangle by 180 degrees, moving
$T$ to the southeast corner.  Then, for $i = 1, 2, \dots, k$, replace each $i$
occurring in $T$ with $k - i + 1$.  Finally, use the jeu-de-taquin sliding
algorithm to move the boxes of $T$ from the southeast corner to the northwest
corner of the rectangle.  The resulting tableau is $e(T)$.

\begin{lem}
Let $\lambda \vdash n$ and let $T \in CST(\lambda, k, \alpha)$ for some $k \geq
0$ and some composition
$\alpha = (\alpha_1, \alpha_2, \dots, \alpha_k)$.\\
1.  $e(T)$ is a well-defined element of $CST(\lambda, k)$, that is, independent
of the choices involved in embedding or sliding.\\
2.  The content of $e(T)$ is $(\alpha_k, \alpha_{k-1}, \dots, \alpha_1)$.\\
3.  The operator $e$ is an involution, that is, $e(e(T)) = T$ always.
\end{lem}

With Part 2 of the above lemma as motivation, given a composition $\alpha =
(\alpha_1, \dots, \alpha_k)$, we define $w_{o_k} \cdot \alpha$ to be the
reverse composition $(\alpha_k, \alpha_{k-1}, \dots , \alpha_1)$.  It follows
from the above lemma that $e$ restricts to an involution on $SYT(\lambda)$ for
arbitrary partitions $\lambda$.  Evacuation interacts nicely with RSK.
Proofs of these results can be found, for example, in \cite{Sag} and
\cite{BBCoxeter}.  

\begin{lem}
Let $w \in S_n$ and suppose $w \mapsto (P,Q)$.\\
1.  $w^{-1} \mapsto (Q,P)$.\\
2.  $w_o w \mapsto (P',e(Q)')$.\\
3.  $w w_o \mapsto (e(P)', Q')$.\\
4.  $w_o w w_o \mapsto (e(P), e(Q))$.\\
5.  $i \in D_L(w)$ if and only if $i \in D(P)$.\\
6.  $i \in D_R(w)$ if and only if $i \in D(Q)$,\\
where $T'$ denotes the conjugate of a standard tableau $T$.
\end{lem}

RSK leads to a natural triple of equivalence relations on $S_n$.  We say that
two permutations $v, w \in S_n$ are \emph{left Knuth equivalent} if we have
$P(w) = P(v)$.  Analogously, $v$ and $w$ are \emph{right Knuth equivalent} if
we have the equality $Q(w) = Q(v)$.  Finally, we have an equivalence relation
given by $w \sim v$ if and only if $w$ and $v$ have the same shape.  Amazingly,
these algorithmic, combinatorial equivalence classes agree with the KL cells.

\begin{thm}
Let $w$ and $v$ be permutations in $S_n$.\\
1.  $w$ and $v$ are left Knuth equivalent if and only if $w$ and $v$ lie in the
same left KL cell.\\
2.  $w$ and $v$ are right Knuth equivalent if and only if $w$ and $v$ lie in
the same right KL cell.\\
3.  $w$ and $v$ have the same shape if and only if $w$ and $v$ lie in the same
two-sided KL cell.
\end{thm}

So, we may interpret the partial orders induced by the left and right KL
preorders as partial orders
on the set of standard tableaux with $n$ boxes.  It can be shown that these two
partial orders
are identical.
Similarly, the induced partial order on two-sided KL cells can be identified
with a partial order
on all partitions of $n$.  It can be shown that this latter partial order is
just the
dominance order $\leq_{dom}$ defined by $\mu \leq_{dom} \lambda$ if and only if
for all $i \geq 0$
we have the comparability of partial sums $\mu_1 + \mu_2 + \cdots + \mu_i \leq
\lambda_1 + \lambda_2 + \cdots + \lambda_i$.
We also have the following change-of-label result, which shows that the
symmetrized KL $\mu$-function behaves well with respect to a translation of
Knuth class.

\begin{thm} (Change of label)
Identify permutations with their images under RSK.  Let $U_1, U_2, T_1,$ and $T_2$ be
standard
tableaux with $n$ boxes, all having the same shape.  We have that\\
1. $\mu[(U_1,T_1),(U_1,T_2)] = \mu[(U_2,T_1),(U_2,T_2)]$ and\\
2. $\mu[(U_1,T_1),(U_2,T_1)] = \mu[(U_1,T_2),(U_2,T_2)]$.
\end{thm}

Thanks to change of label, we can define $\mu[P,Q]$ for two standard tableaux
$P$ and $Q$ of the same shape to be the common value of $\mu[(P,T),(Q,T)] =
\mu[(T,P),(T,Q)]$, where $T$ is any standard tableau having the same shape as
$P$ and $Q$.

With these results in hand, we are ready to define a powerful avatar of the
Specht modules.  Let $\lambda \vdash n$ be a partition and choose an arbitrary
tableau $T \in SYT(\lambda)$.  By Theorem 2.7 and its following remarks
we have a left action of $S_n$ on
$S_0^{\lambda,T} := \bigoplus_w \mathbb{C} \{C'_w(1)\}$, where the $w$ in the
direct sum ranges over all permutations such that either sh($T$) $<_{dom}$
sh($w$) or $T = P(w)$.  Also by the paragraph following Theorem 2.7 the module $S_0^{\lambda, T}$
contains an $S_n$-invariant submodule $S_1^{\lambda,T} := \bigoplus_w
\mathbb{C} \{ C'_w(1) \}$, where now $w$ ranges over all permutations such that
sh($T$) $<_{dom}$ sh($w$).  It therefore makes sense to define $S^{\lambda,T}$
to be the quotient module
\begin{equation*}
S^{\lambda,T} := S_0^{\lambda,T} / S_1^{\lambda,T}.
\end{equation*}
\noindent
The vector space $S^{\lambda,T}$ carries an action of the group algebra $\csn$.
We can identify the basis vectors of $S^{\lambda,T}$ with elements of
$SYT(\lambda)$ via associating to a tableau $U$ of shape $\lambda$ the image of
$C'_{(T,U)}(1)$ in the above quotient.

It is natural to ask to what degree the module $S^{\lambda,T}$ depends on the
choice of standard tableau $T$.
In this direction, it turns out that the action of the Coxeter generators $s_i$
on the space $S^{\lambda, T}$ is completely determined by the symmetrized KL
$\mu$-function and the purely combinatorial descent set of a tableau.
Identifying the basis elements of $S^{\lambda,T}$ with tableaux in
$SYT(\lambda)$, we have that the action of $s_i$ on $S^{\lambda, T}$ is given
by the explicit formula
\begin{equation}
s_i P =
\begin{cases}
-P & \text{if $i \in D(P)$} \\
P + \sum_{i \in D(Q)} \mu[P,Q] Q & \text{if $i \notin D(P)$.}
\end{cases}
\end{equation}
This formula can be proven via some reasonably explicit manipulations of the KL
basis and KL polynomials \cite{GarMcL}.
In particular, the action of $s_i$ does not depend on the tableau $T$ that we
chose in the construction of $S^{\lambda, T}$.  It follows that the matrices
giving the left action of $s_i$ with respect to the given bases of $S^{\lambda,
T}$ and $S^{\lambda, U}$ are literally equal (up to reordering basis elements)
for any two standard tableaux $T, U \in SYT(\lambda)$.

With this strong isomorphism in hand, we define the $\csn$-module $S^{\lambda}$
to be the $\csn$-module $S^{\lambda, T}$ for any choice of $T \in
SYT(\lambda)$.  Again remarkably, the modules $S^{\lambda}$ are precisely the
irreducible representations of $S_n$.  For an exposition and extended version
of the following result, see \cite{GarMcL}.

\begin{thm}
The module $S^{\lambda}$ is isomorphic as a $\csn$-module to the irreducible
representation of $S_n$ indexed by the partition $\lambda$.
\end{thm}

$S^{\lambda}$ is called the \emph{(left) KL cellular representation} indexed by
$\lambda$.  The basis elements of $S^{\lambda}$ are homomorphic images of a
subset of the KL basis of $\csn$ and are in natural bijection with the elements
of $SYT(\lambda)$.  This avatar of the Specht modules will turn out to be very
useful in the representation theoretic modeling of combinatorial operators.

An earlier use of the modules $S^{\lambda}$ in this way is due to Berenstein,
Zelevinsky \cite{BerZelCan} and Stembridge \cite{StemTab} and concerns the
action of the evacuation operator $e$.  Since $e$ has order 2, one would hope
to find an order 2 element in the symmetric group $S_n$ which maps (up to sign)
to the permutation matrix corresponding the evacuation under the KL cellular
representation.  It turns out that the long element does the trick.

\begin{thm}$($Berenstein-Zelevinsky \cite{BerZelCan}, Stembridge
\cite{StemTab}$)$  Identify the basis vectors of the Kazhdan-Lusztig (left)
cellular representation corresponding to $\lambda$ with elements $P \in
SYT(\lambda)$.  Denote this representation by $\rho: S_n \rightarrow
GL(S^{\lambda})$.  Let $w_o \in S_n$ be the long element.

Then, up to a plus or minus sign, we have that
$\rho(w_o)$ is the linear operator which sends $P$ to $e(P)$, where $e$ is
evacuation.
\end{thm}

Informally, this result states that given any partition $\lambda \vdash n$, the
image of the long element $w_o \in S_n$ models the action of evacuation on
$SYT(\lambda)$.  It will turn out that for \emph{rectangular} partitions
$\lambda$, the image of the long \emph{cycle} $(1, 2, \dots, n)$ in $S_n$ will
model the action of promotion on $SYT(\lambda)$.  A straightforward application
of the Murnaghan-Nakayama rule determines whether the sign appearing in the
above theorem is a plus or a minus.  The issue of resolving this sign will be
slightly more involved for us and will be taken care of by direct manipulation
of the KL basis.  In the next section we show precisely how this is done.

\section{Promotion on Standard Tableaux}

Given a rectangular shape $\lambda = b^a$ with $ab = n$, we want to determine
how the operation of jeu-de-taquin promotion on
$SYT(\lambda)$ interacts with the left KL cellular representation
$S^{\lambda}$ of shape $\lambda$.  Our first goal is to show that the promotion
operator $j$ interacts nicely with the $\mu$ function.

Define a deletion operator
$d: SYT(b^a) \rightarrow SYT( b^{a-1} (b-1))$
by letting $d(U)$ be the (standard) tableau obtained by deleting the $n$ in the
lower right hand corner of $U$.  It is easy to see that $d$ is a bijection.
Our first lemma shows that $d$ is well behaved with respect to the $\mu$
function.
\begin{lem}
For any $U, T \in SYT(b^a)$, we have that
$\mu[U,T] = \mu[d(U), d(T)]$.
\end{lem}
\begin{proof}
The line of reasoning here is similar to one used by Taskin in  Lemma 3.12 of
\cite{TaskOrders}.

Recall that for any partition $\lambda \vdash n$ we have the associated column
superstandard tableau
$CSS(\lambda)$ defined by inserting the numbers $1, 2, \dots, n$ into the
diagram of $\lambda$ first from top to bottom within each column, then from
left to right across columns.
Define permutations $u = u_1 \dots u_n$,
$v = v_1 \dots v_n \in S_n$ and
$t = t_1 \dots t_{n-1}$, $w = w_1 \dots w_{n-1} \in S_{n-1}$ by their images
under RSK:
\begin{align*}
u & \mapsto (U, CSS(b^a)),\\
v & \mapsto (T, CSS(b^a)),\\
w & \mapsto (d(U), CSS(b^{a-1}(b-1)),\\
t & \mapsto (d(T), CSS(b^{a-1}(b-1)).
\end{align*}
Using the definition of the RSK algorithm it's easy to check that, using
one-line notation for permutations in $S_n$,
\begin{align*}
u_1 \dots u_n &= w_1 w_2 \dots w_{n-a}(n)w_{n-a+1} \dots w_{n-1},\\
v_1 \dots v_n &= t_1 t_2 \dots t_{n-a}(n)t_{n-a+1} \dots t_{n-1}.
\end{align*}
By a result of Brenti (\cite{BrentiCPKL}, Theorem 4.4), we have an equality of
Kazhdan-Lusztig polynomials $P_{w,t}(q) = P_{w(n), t(n)}(q)$.  For
$1 \leq k \leq n$, define permutations $w^{(k)}$ and $t^{(k)}$ in $S_n$ by
$w^{(k)} := s_k s_{k+1} \cdots s_{n-1} (w(n))$,
$t^{(k)} := s_k s_{k+1} \cdots s_{n-1} (t(n))$.
Here $w(n)$ (respectively $t(n)$) is the permutation in $S_n$ whose one-line notation is $w_1 \dots w_{n-1} n$ (respectively 
$t_1 \dots t_{n-1} n$). 
By the above remarks it follows that $w^{(n-a+1)} = u$ and
$t^{(n-a+1)} = v$.

We claim that for all $k$,
$P_{w^{(k)}, t^{(k)}}(q) = P_{w^{(k+1)}, t^{(k+1)}}(q)$.  It will then follow
by induction and Brenti's result that
$P_{u,v}(q) = P_{w,t}(q)$.  To see this, notice that
$w^{(k)} = s_k w^{(k+1)}$, $t^{(k)} = s_k t^{(k+1)}$, and the transposition
$s_k$ satisfies
$s_k \in D_{L}(w^{(k+1)}) \cap D_{L}(t^{(k+1)})$.  These conditions together
with Lemma 2.3 in the case $c = 1$ imply that we have the following polynomial
relation:
\begin{multline}
P_{w^{(k)}, t^{(k)}}(q) =
P_{s_k w^{(k+1)}, s_k t^{(k+1)}}(q) \\
= P_{w^{(k+1)}, t^{(k+1)}}(q) +
q P_{s_k w^{(k+1)}, t^{(k+1)}}(q) -
\sum_{s_k r < r} q^{\frac{\ell(r, t^{(k)})}{2}} \mu(r, t^{(k+1)})
P_{s_k w^{(k+1)}, r}(q).
\end{multline}
However, we also have that both $w^{(i-1)}$ and $t^{(i-1)}$ map $i$ to $n$ for
all $i$.  Therefore, $w^{(k)} \nleq t^{(k+1)}$ and the Bruhat interval
$[w^{(k)}, t^{(k+1)}]$ is empty.
This implies that the only surviving term in the above sum is $P_{w^{(k+1)},
t^{(k+1)}}(q)$ and we get that $P_{u,v}(q) = P_{w,t}(q)$, as desired.  To
complete the proof, we make the easy observations that $\ell(u,v) = \ell(w,t)$
and that the pairs $(u, v)$ and $(w, t)$ both lie in the same right cell.
Therefore, by taking the left tableaux of the pairs $(u, v)$ and $(w, t)$,  we
get the desired equality of $\mu$ coefficients.
\end{proof}
Observe that the proof of Lemma 3.1 shows a stronger statement regarding an
equality of KL polynomials rather than just an equality of their top
coefficients.

It is easy to show using Lemma 2.4 and Part 4 of Lemma 2.6 that
$\mu[e(P), e(Q)] = \mu[P, Q]$ for any standard tableaux $P$ and $Q$ having the
same \emph{arbitrary} shape.  We use this fact and the above result to get the
desired fact about the action of $j$.
\begin{prop}
Let $P, Q$ be standard tableaux which are either both in $SYT(b^a)$ or
$SYT(b^{a-1}(b-1))$.  We have that
$\mu[P, Q] = \mu[j(P), j(Q)]$.
\end{prop}

This proposition says that, in the special case that our shape is a rectangle
or a rectangle missing its outer corner, the action of promotion preserves the
$\mu$ function.  This does not hold in general for other shapes.  Indeed, a
counterexample may be found in the smallest shapes which do not meet the
criteria of Proposition 3.2: $\lambda = (3,1)$ and $\lambda = (2,1,1)$.  In the
case of $\lambda = (3,1)$, promotion acts on the set $SYT((3,1))$ via a single
$3$-cycle:
\begin{equation*}
\left(
\begin{array}{ccccccccccc}
1 & 2 & 3 &   & 1 & 3 & 4 &   & 1 & 2 & 4 \\
4 &   &   & , & 2 &   &   & , & 3 &   &
\end{array} \right).
\end{equation*}
However, it is easy to check that in $S_4$ we have
$\mu[4123,2134] = 0$, $\mu[2134,3124] = 1$, and
$\mu[3124, 4123] = 1$.  Since for any standard tableau $T$ the column reading
word of $T$ (i.e., the element of $S_{|T|}$ obtained by reading off the letters
in $T$ from bottom to top, then from left to right) has insertion tableau $T$,
we obtain the claimed counterexample.

This lemma and the above example give some representation theoretic
justification for why we need our hypotheses that $\lambda$ be rectangular in
Theorems 1.5 through 1.7, but perhaps leave mysterious why these results fail
in general for rectangles minus outer corners.  We will soon give combinatorial
intuition for why the strictly rectangular hypotheses are needed.
\begin{proof} (of Proposition 3.2)
We prove this first for the case of rectangular shapes.  The case of shapes
which are rectangular minus an outer corner will follow easily.
We introduce a family of operators on tableaux.  Define a creation operator
$c: SYT(b^{a-1}(b-1)) \rightarrow SYT(b^a)$ by letting $c(T)$ be the tableau
obtained by adding a box labeled $n$ to the lower right hand corner of $T$.
Given a tableau $T$ of shape $\lambda / \mu$, let $r(T)$ be the tableau
obtained by rotating $T$ by $180$ degrees and let $\delta(T)$ be the tableau
obtained by replacing every entry $i$ in $T$ with
$|T| + 1 - i$.  Given a standard tableau $T$ which is not skew, let $J(T)$ be
the tableau obtained by embedding $T$ in the northwest corner of a very large
rectangle, and playing jeu-de-taquin to move the boxes of $T$ to the southeast
corner.
Observe that we have the following relations of operators:
\begin{align*}
\text{I.}  \hspace{5pt}& \delta r  = r \delta \\
\text{II.}  \hspace{5pt}&  e         = \delta r J \\
\text{III.}  \hspace{5pt}& j         = r \delta c \delta r J d
\end{align*}
Let $P$ and $Q$ be standard tableaux of shape $b^a$.
With the above definitions, we have the following chain of equalities:
\begin{align*}
\mu[P, Q]  &=  \mu[d(P), d(Q)]  & \text{(Lemma 3.1)} \\
           &=  \mu[ed(P), ed(Q)] & \\
           &=  \mu[\delta rJd(P), \delta rJd(Q)] & \text{(II)} \\
&= \mu[c \delta rJd(P), c \delta rJd(Q)] & \text{(Lemma 3.1)}\\
&= \mu[ec \delta rJd(P), ec \delta rJd(Q)] & \\
&= \mu[\delta rJc \delta rJd(P), \delta rJc \delta rJd(Q)] &
\text{(II)} \\
&= \mu[\delta rc \delta rJd(P), \delta rc \delta rJd(Q)] &
\text{$(c \delta rJd(P)$, $c \delta rJd(Q)$ are rectangles)}\\
&= \mu[r \delta c \delta rJd(P), r \delta c \delta rJd(Q)]
& \text{(I)} \\
&= \mu[j(P), j(Q)]. & \text{(III)}
\end{align*}
This completes the proof for the case of rectangular shapes.
For the case of shapes which are rectangular minus an outer corner, notice that
we have the following equality of operators on $SYT(b^{a-1}(b-1))$:
\begin{align*}
 \text{IV.} \hspace{5pt}& j = d j c.
\end{align*}
Let $P$ and $Q$ be in $SYT(b^{a-1}(b-1))$.  We have the following chain  of
equalities:
\begin{align*}
\mu[P,Q] &= \mu[c(P),c(Q)] & \text{(Lemma 3.1)}\\
&= \mu[jc(P),jc(Q)] & \text{($c(P)$, $c(Q)$ are rectangles)}\\
&= \mu[djc(P),djc(Q)] & \text{(Lemma 3.1)}\\
&= \mu[j(P),j(Q)]. & \text{(IV)}
\end{align*}

This completes the proof of the lemma.

\end{proof}
To better understand the action of promotion, we introduce a new combinatorial
set related to rectangular tableaux which will be the ordinary tableau descent
set with the possible addition of $n$.
Given a standard tableau $P$ of rectangular shape with $n$ boxes, define the
\emph{extended descent set of P} denoted $D_e(P) \subseteq [n]$, as follows.
For $i = 1, 2, \dots, n-1$, say that $i \in D_e(P)$ if and only if $i$ is
contained in the ordinary descent set $D(P)$.  To determine if $n$ is contained
in $D_e(P)$, consider the tableau $U$ with entries $\{ 2, 3, \dots, n \}$
obtained by deleting the $1$ in $P$ and playing jeu-de-taquin to move the
resulting hole to the southeastern corner.  The entry $n$ is either immediately
north or immediately west of this hole in $U$.  Say that $n$ is contained in
$D_e(P)$ if and only if $n$ appears north of this hole in $U$.   The
fundamental combinatorial fact about the extended descent set is that promotion
acts on it by cyclic rotation.
\begin{lem}
Let $P \in SYT(b^a)$.  For any $i$, we have that $i \in D_e(P)$ (mod $n$) if
and only if $i+1 \in D_e(j(P))$ (mod $n$).
\end{lem}
\begin{ex}
Let $b = 4$ and $a = 3$ and let $P$ be the following element of $SYT(4,4,4)$:
\begin{equation*}
\begin{array}{cccccc}
     & 1 & 2 & 4  &  9 &  \\
P =  & 3 & 5 & 8  & 11 &  \\
     & 6 & 7 & 10 & 12 &.
\end{array}
\end{equation*}
Applying the promotion operator to $P$ yields
\begin{equation*}
\begin{array}{ccccccccccccccc}
        & 1 & 2 & 3  & 5  & \\
j(P) =  & 4 & 6 & 9  & 10 & \\
        & 7 & 8 & 11 & 12 &.
\end{array}
\end{equation*}
It can be shown that $D_e(P) = \{2, 4, 5, 9, 11 \}$ and $D_e(j(P)) = \{3, 5, 6,
10, 12 \}$.
\end{ex}
For partitions $\lambda \vdash n$ other than rectangles, it is not in general
possible to define an extended descent set which agrees with the ordinary
descent set on the letters $[n-1]$ and on which $j$ acts with order $n$.
Indeed, it can be shown that promotion acts with order $6$ on $SYT((3,2))$ and
$(3,2)$ is a partition of $5$.  This gives some combinatorial intuition for why
we needed the strictly rectangular hypotheses in our results.
\begin{proof}
We prove the equivalent statement involving jeu-de-taquin demotion.  Let $Q :=
j^{-1}(P)$ denote the image of $P$ under the demotion operator and suppose that
$i \in D_e(P)$ for $i \in [n]$.  We want to show that $i-1 \in D_e(Q)$ (mod
$n$).  This is shown in several cases depending on the value of $i$.  Let $U$
be as in the above paragraph and let $\Gamma$ be the jeu-de-taquin path in the
rectangle $b^a$ involved in obtaining $U$ from $P$.

\emph{Case I.}  $i \in \{2, 3, \dots, n-1 \}$.

For this case, observe that a single jeu-de-taquin slide does not affect the
(ordinary) descent set of a tableau.  Here we extend the definition of the
descent set of a tableau to a tableau with a hole in it in the obvious way.

\emph{Case II.} $i = n$.

Since $P$ is rectangular, $\Gamma$ must contain $n$.  Since $n$ is in $D_e(P)$,
the entry immediately above $n$ is also contained in $\Gamma$.  Thus, $n$ moves
up when $U$ is constructed and we get that $n-1$ is immediately above $n$ in
$Q$.  Therefore, $n-1$ is a descent of $Q$, as desired.

\emph{Case III.} $i = 1$.

We must show that $n \in D_e(Q)$.  To do this, we consider the action of
demotion on $Q$.  Let $\Delta$ be the jeu-de-taquin path in $Q$ involved in
sliding out a hole at the upper left, so that $\Delta$ ends in the lower right
hand corner of $Q$ at the entry $n$.  We want to show that the final edge of
$\Delta$ is a downward edge into $n$.

To do this, define a subset $S$ of the path $\Gamma$ by 
$(x,y) \in S$ if and only if $(x+1,y) \in \Gamma$.  Here we consider
diagrams of partitions to be subsets of the fourth quadrant of
the plane intersected with the integer lattice $\mathbb{Z}^2$ which have 
their upper leftmost coordinate equal to $(1,-1)$.  Since $\Gamma$ starts
at $(1,-1)$ and ends at $(b,-a)$, the projection of 
$S$ onto the $x$-axis is the full interval $[1,b-1]$.  If the final edge
of $\Delta$ were not a downward edge into $n$, it would follow that
$\Delta$ intersected the set $S$ nontrivially.  Choose $(x,y)$ with $x$
minimal so that $(x,y) \in \Delta \cap S$.  

We claim that $(x,y+1)$ is not contained in $\Delta$.
Indeed, if $(x,y+1)$ were contained in $\Delta$, then since 
$\Delta$ is a jeu-de-taquin path the $(x,y)$-entry of $Q$ is less than
the $(x+1,y+1)$-entry of $Q$, contradicting the fact that $P$ is standard.  So, either $x = 1$ or $(x-1,y) \in \Delta$.  If $x = 1$, then since $(1,y) \in \Delta$ but $(1,y+1) \notin \Delta$, we must
have that $y = -1$.  But $(1,-1) \in S$ contradicts $1$ being a descent
of $P$.  If $(x-1,y) \in \Delta$, then the fact that the projection of
$S$ onto the $x$-axis is $[1,b-1]$ contradicts the minimality of 
$x$.  We conclude that the final edge of $\Delta$ is a downward edge 
into $n$, as desired.
\end{proof}
We prove a technical lemma about the image of the long cycle under the KL
representation.  This lemma will be used to pin down a constant in a Schur's
Lemma argument in the proof of Proposition 3.5.
\begin{lem}
Let $c_n = (1, 2, \dots, n) \in S_n$ be the long cycle.  Let $\lambda = b^a$ be
a rectangle.  Identify permutations with their images under RSK.  The
coefficient of
the KL basis element $C'_{(j(CSS(\lambda)), CSS(\lambda))}(1)$ in the expansion
of $C'_{(CSS(\lambda), CSS(\lambda))}(1)c_n$ in the KL basis of $\csn$ is
$(-1)^{a-1}$.
\end{lem}
\begin{proof}
Define permutations $u, v \in S_n$ by
\begin{align*}
u & \mapsto (CSS(\lambda), CSS(\lambda)), \\
v & \mapsto (j(CSS(\lambda)), CSS(\lambda)).
\end{align*}
It follows that the one-line notation for $u$ and $v$ is

\begin{align*}
u = & a(a-1) \dots 1(2a)(2a-1) \dots (a+1) \dots n \dots (n-a+1),\\
v = & (a+1)a \dots 31(2a+1)(2a) \dots (a+3)2
     (3a+1)(3a) \dots (2a+3)(a+2) \dots\\
     & n(n-1) \dots
     (n-a+2)(n+2-2a).
\end{align*}

By inspection, both of these are 3412 and 4231 avoiding, so therefore by
smoothness considerations
(see Theorem A2 of \cite{KLRepCH}) all of the KL polynomials $P_{x,u}(q)$ for
$x \in S_{a^b}$ are equal to $1$ and we have the corresponding formulas:

\begin{align*}
C'_u(1) &= \sum_{x \in S_{a^b}} (-1)^{\ell(x,u)} x,\\
C'_u(1)c_n &= \sum_{x \in S_{a^b}} (-1)^{\ell(x,u)} xc_n.
\end{align*}

The unique Bruhat maximal element $x_o c_n$ for which $x_o \in S_{a^b}$ has
one-line notation
\begin{equation*}
x_o c_n = (a+1)a \dots 32(2a+1)(2a) \dots (a+2)(3a+1) \dots
(2a+2) \dots n(n-1) \dots (n-a+2)1.
\end{equation*}
Right multiplying by $c_n^{-1}$ yields that
\begin{equation*}
x_o = a(a-1) \dots 21(2a) \dots
      (n-a)(n-a-1) \dots (n-2a+1)(n-1)(n-2) \dots (n-a+1)(n).
\end{equation*}
Therefore, $x_o$ and $u$ differ by a cycle of length $a$ and
$(-1)^{\ell(x_o, u)} = (-1)^{a-1}$.  It follows that there exist numbers
$\gamma_y \in \mathbb{C}$,
\begin{equation*}
C'_u(1) c_n = (-1)^{a-1} C'_{x_o c_n}(1) +
\sum_{y < x_o c_n} \gamma_y C'_y(1).
\end{equation*}
The coefficient we are interested in is $\gamma_v$.  First, notice that the
Bruhat interval $[v, x_o c_n]$ consists of all permutations in $S_n$ whose one
line notation has the form
\begin{equation*}
(a+1)a \dots 3 \beta_1 (2a+1)(2a) \dots (a+3) \beta_2 \dots
n(n-1) \dots (n-a+2) \beta_b,
\end{equation*}
where $\beta_1 \in \{ 1, 2 \}$, $\beta_2 \in \{2, a+2 \}$,
\dots $\beta_{b-1} \in \{n-3a+2, n-2a+2 \}$,
$\beta_b \in \{n-2a+2, 1 \}$.  It follows from the definition of the Bruhat
order that the subposet of such permutations is isomorphic to the Boolean
lattice $B_{b-1}$ of rank $b-1$.  Moreover, all permutations of the above form
are 3412 and 4231 avoiding, meaning that for every $w \in [v, x_o c_n]$ we have
that $C'_w(1) = \sum_{z \leq w} (-1)^{\ell(z,w)} z$.  Finally, we observe that
$[v, x_o c_n] \cap S_{a^b} c_n = \{ x_o c_n \}$.

All of the conditions in the above paragraph imply that $\gamma_v =
(-1)^{a-1}$, as desired.
\end{proof}
We now have all of the ingredients necessary to analyze the relationship
between $j$ and the KL representation.  Specifically, we show that up to a
predictable scalar, $j$ acts like the long cycle $(1, 2, \dots, n) \in S_n$.
This gives our desired analogue of Equation 1.5.

\begin{prop}
Let $\lambda = b^a$ be a rectangular shape with $ab = n$ and let
$\rho: S_n \rightarrow GL(S^{\lambda})$ be the associated KL cellular
representation, with basis identified with $SYT(\lambda)$.  Let $c_n$ denote
the long cycle $(1, 2, \dots n) \in S_n$.  Define a $\mathbb{C}$-linear map
$J: S^{\lambda} \rightarrow S^{\lambda}$ by extending
$J(P) = j(P)$.  We have that
\begin{equation}
\rho(c_n) = (-1)^{a-1} J.
\end{equation}
Equivalently, for any standard tableau $P \in SYT(\lambda)$, we have that
\begin{equation}
\rho(c_n)(P) = (-1)^{a-1} j(P).
\end{equation}
\end{prop}
\begin{proof}
We want to show that the operator $J^{-1} \rho(c_n)$
commutes with the action of $S_n$.  If we can do this, by Schur's Lemma it will
follow that for some constant $\gamma \in \mathbb{C}$ we have the equality of
operators $J = \gamma \rho(c_n)$.  In \cite{StemTab}, Stembridge used
this sort of argument to prove Theorem 2.10.

Since $S_n$ is generated by the simple reflections $s_i$ for $i = 1, 2, \dots,
n-1$, it is enough to show that $J^{-1} \rho(c_n)$ commutes with $\rho(s_i)$
for each $i = 1, 2, \dots, n-1$.  From conjugation within the symmetric group
we know that $\rho(c_n) \rho(s_i) = \rho(s_{i+1}) \rho(c_n)$ always, where we
interpret $s_n$ to be the `affine' transposition $(1,n)$.  So, it is enough to
show that $J^{-1} \rho(s_{i+1}) = \rho(s_i) J^{-1}$ always.  For $i = 1, 2,
\dots, n-2$ this will follow from the action of Coxeter generators on
$S^{\lambda}$, but for $i = n-1$ this poses a problem since there is no known
nice formula in general for the action of $s_n$ on these modules.  In fact, we
show that $J^{-1} \rho(c_n)$ commutes with $\rho(s_i)$ for only $i = 1, 2,
\dots, n-2$ and then apply the branching rule for restriction of irreducible
symmetric group modules to show that these commutation relations are in fact
sufficient for our purposes.

We first show that for $i = 1, 2, \dots, n-2$ we have
\begin{equation*}
J^{-1}\rho(c_n)\rho(s_i) = \rho(s_i) J^{-1} \rho(c_n).
\end{equation*}
Given a standard tableau $P$, let $j^{-1}(P)$ be the unique tableau which maps
to $P$ under promotion.
From the corresponding conjugation relation in the symmetric group it follows
that
\begin{equation*}
\rho(c_n) \rho(s_i) = \rho(s_{i+1}) \rho(c_n).
\end{equation*}
On the other hand, because $1 \leq i \leq n-2$ and $J$ acts as $j$ on basis
elements indexed by standard tableaux, we have that for any $P \in
SYT(\lambda)$,
\begin{align*}
J^{-1} \rho(s_{i+1}) (P) &=
\begin{cases}
-j^{-1}(P)  &\text{if $i+1 \in D_e(P)$} \\
j^{-1}(P) + \sum_{i+1 \in D_e(Q)} \mu[P,Q] j^{-1}(Q)
&\text{if $i+1 \notin D_e(P)$}\\
\end{cases}\\
&=
\begin{cases}
-j^{-1}(P) &\text{if $i \in D_e(j^{-1}(P))$}\\
j^{-1}(P) + \sum_{i \in D_e(j^{-1}(Q))} \mu[P,Q] j^{-1}(Q)
&\text{if $i \notin D_e(j^{-1}(P))$}\\
\end{cases}\\
&=
\begin{cases}
-j^{-1}(P) &\text{if $i \in D_e(j^{-1}(P))$}\\
j^{-1}(P) + \sum_{i \in D_e(j^{-1}(Q))} \mu[j^{-1}(P),j^{-1}(Q)] j^{-1}(Q)
&\text{if $i \notin D_e(j^{-1}(P))$}\\
\end{cases}\\
&= \rho(s_i) J^{-1}(P).
\end{align*}

The first equality is the action of the simple transpositions on the cellular
KL basis, the second is Lemma 3.3, the third is Proposition 3.2, and the fourth
is again the definition of the action of simple transpositions.

The above discussion implies that the operator $J^{-1} \rho(c_n)$ commutes with
the action of the parabolic subgroup $S_{n-1}$ of $S_n$ on the irreducible
$S_n$-module $S^{\lambda}$.
Since $\lambda$ is a rectangle, $\lambda$ has a unique outer corner and by the
branching rule for symmetric groups, the restriction
$S^{\lambda}\downarrow^{S_n}_{S_{n-1}}$ remains an irreducible
$S_{n-1}$-module.  Therefore, by Schur's Lemma, we conclude that there is a
number $\gamma \in \mathbb{C}$ so that
\begin{equation*}
J = \gamma \rho(c_n).
\end{equation*}
We want to show that $\gamma = (-1)^{a-1}$.
This follows from Lemma 3.4.
\end{proof}
The above result states that for Specht modules of rectangular shape the image
of the long cycle $c_n$ under the KL representation is plus or minus the
permutation matrix which encodes jeu-de-taquin promotion.  It is not true that
any conjugate of $c_n$ in $S_n$ enjoys this property - indeed, many are not
even permutation matrices.  Observe the analogy to Theorem 2.10 which states
that the image of the long \emph{element} $w_o$ under the KL representation for
an arbitrary shape is plus or minus a permutation matrix which encodes
evacuation.

Also notice that if $\lambda$ is a partition of any nonrectangular shape, then
$\lambda$ has more than one outer corner.  This implies that the restricted
module
$S^{\lambda} \downarrow^{S_n}_{S_{n-1}}$ is not irreducible and the above proof
breaks down.

As a corollary to this we get a classical result on the order of promotion
\cite{Haiman}.
\begin{cor}
For $\lambda \vdash n$ rectangular, every element of $SYT(\lambda)$ is fixed by
$j^n$.
\end{cor}

\begin{rem}
For arbitrary partitions $\lambda$ it is not true that the order of promotion
on $SYT(\lambda)$ divides $|\lambda|$.  In fact, this order for arbitrary
$\lambda$ is unknown.  So, the hypothesis that $\lambda$ be rectangular in the
statement of the the corollary is necessary.  Moreover,
it is not in general true that for a rectangular partition $\lambda \vdash n$
the order of promotion on $SYT(\lambda)$ is equal to $n$.  For an extreme
counterexample, if $\lambda$ is either a single row or a single column, the set
$SYT(\lambda)$ consists of a single element and $j$ has order $1$.  Other
counterexamples exist, minimal among them $\lambda = (2,2)$.  The author does
not know of a formula for the order of $j$ on $SYT(\lambda)$ for arbitrary
shapes $\lambda$.  As a possible indication of the subtlety here, given
\emph{any} nonrectangular partition $\lambda \vdash n$, there exists a tableau
$P \in SYT(\lambda)$ which is not fixed by the $n^{th}$ power of promotion
\cite{PRPromo}.

For the special case of staircase shapes $\nu = (k, k-1, \dots, 1) \vdash n$,
the operator 
$j^{2n}$ fixes every element of $SYT(\nu)$.  This leads one to hope that
the KL cellular basis may be exploited in studying the action of promotion on
$SYT(\nu)$ for staircase shapes $\nu$.  Unfortunately, even for the staircase
$(3,2,1)$ we do not have the result analogous to Proposition 3.2 that $j$
preserves the $\mu$ function.  It is, however, possible to show that for any
staircase $\nu$, the action of $j^{|\nu|}$ on $SYT(\nu)$ is given by
transposition, which is known to preserve the $\mu$ function.  Taken together,
the operators of transposition and evacuation generate a group isomorphic to
the Klein four group $\mathbb{Z} / 2 \mathbb{Z} \times \mathbb{Z} / 2
\mathbb{Z}$, leaving some hope that the study of the action of this smaller
group might be analyzed by our methods.  Unfortunately, the case of the
staircase $(3,2,1)$ also shows that no power $j^d$ for $1 \leq d \leq 5$
preserves the $\mu$ function, so it would seem difficult to use our methods to
analyze the action of any of the larger groups on the set $SYT(\nu)$ generated
by lower positive powers of $j$ together with evacuation.

Even in the case of rectangular shapes $\lambda$, the previous corollary does
not give the order of $j$ on a specific tableau $T \in SYT(\lambda)$.  For the
case where $\lambda$ has $\leq 3$ rows, this order can be interpreted as the
order of the cyclic symmetry of a combinatorial object associated to $T$ called
an $A_2$ web \cite{PWeb}.
\end{rem}

Let $s_n$ denote the ``affine'' transposition $(1,n)$ in $S_n$.
As another corollary, we get a nice formula for the image of $s_n$ under the KL
representation.

\begin{cor}
Let $\lambda = b^a$ be a rectangular shape and identify the basis of the
corresponding left KL cellular representation with $SYT(\lambda)$.

We have that
\begin{equation*}
s_n P =
\begin{cases}
-P & \text{if $n \in D_e(P)$}\\
P + \sum_{n \in D_e(Q)} \mu[P,Q] Q & \text{if $n \notin D_e(P)$}.
\end{cases}
\end{equation*}
\end{cor}
\begin{proof}
Using the notation of the proof of the above proposition, we have that
$s_{n-1}$ commutes with the action of
$J^{-1} \rho(c)$.  It follows that
$J^{-1} s_n (P) = s_{n-1} J^{-1} (P)$.  The desired formula follows.
\end{proof}

Finally, the above analysis leads to many equalities of $\mu$ coefficients.

\begin{cor}
Let $\lambda \vdash n$ be a rectangle and let $Y$ and $Y'$ be two orbits of
$SYT(\lambda)$ under the action of $j$ which have relatively prime orders.
Given $P, Q \in Y$ and $P', Q' \in Y'$ we have that $\mu[P,P'] = \mu[Q,Q']$.
\end{cor}

\begin{ex}
For $n = 6$ and $\lambda = (2,2,2)$, we have that $|SYT(\lambda)| = 5$ and the
cyclic action of $j$ on $SYT(\lambda)$ breaks $SYT(\lambda)$ into one cycle of
size $2$ and one cycle of size $3$.  Explicitly, the operator $j$ acts on
$SYT((2,2,2))$ via:
\begin{align*}
\left(
\begin{array}{cccccccc}
1 & 4 &   & 1 & 2 &   & 1 & 3 \\
2 & 5 &   & 3 & 5 &   & 2 & 4 \\
3 & 6 & , & 4 & 6 & , & 5 & 6
\end{array}
\right) &
\left(
\begin{array}{ccccc}
1 & 3 &   & 1 & 2 \\
2 & 5 &   & 3 & 4 \\
4 & 6 & , & 5 & 6
\end{array}
\right)
\end{align*}
Mapping the elements of these orbits to their column reading words in $S_6$,
the above Corollary implies that when $v$ is any element of the set $\{ 321654,
521643, 431652 \} \subset S_6$ and $w$ is any element of the set $\{ 421653,
531642 \}$ we have that $\mu[ v,w ]$ is a constant.  Since $421653$ covers
$321654$ in Bruhat order, we see that this common value of $\mu$ is equal to
$1$.
\end{ex}

An application of the above proposition yields
Theorem 1.3, our desired result on cyclic
sieving in the action of jeu-de-taquin on standard tableaux.
Implicit in the statement of 
Theorem 1.3 is the fact that the $n^{th}$ power of
promotion fixes every element of $SYT(\lambda)$.  This fails if $\lambda$ is
nonrectangular.  Even if we were to replace $C$ with the cyclic group having
order equal to the order of promotion on $SYT(\lambda)$ for nonrectangular
$\lambda$, the action of promotion together with the appropriate $q$-hook
length formula would still not in general yield a CSP.  For example, consider
the shape $\lambda = (3,3,1) \vdash 7$.  There are 21 standard tableaux of
shape $\lambda$ and the action of promotion breaks $SYT(\lambda)$ up into three
orbits having sizes 3, 5, and 13.  Thus, the order of promotion on
$SYT(\lambda)$ is $3*5*13 = 195$.  It can be shown that the $q$-hook length
formula for the shape lambda is $X(q) = (1 + q + q^2 + q^3 + q^4 + q^5 + q^6)(1
+ q^2 + q^4)$.  The evaluation of $X(q)$ at a primitive $195^{th}$ root of
unity is not even real, let alone equal to zero, the number of fixed points of
one iteration of promotion on $SYT(\lambda)$.

\begin{ex}
In the above example of the action of promotion on $SYT((2,2,2))$, we have that
\begin{equation*}
f^{(2,2,2)}(q) =
\frac{[6]!_q}{[4]_q[3]_q[2]_q[3]_q[2]_q[1]_q} =
(1 - q + q^2)(1 + q + q^2 + q^3 + q^4).
\end{equation*}
Letting $\zeta = e^{\frac{\pi i}{3}}$, we compute directly that
\begin{equation*}
\begin{array}{ccc}
f^{(2,2,2)}(1) = 5 & f^{(2,2,2)}(\zeta) = 0
& f^{(2,2,2)}(\zeta^2) = 2 \\
f^{(2,2,2)}(\zeta^3) = 3 &
f^{(2,2,2)}(\zeta^4) = 2 &
f^{(2,2,2)}(\zeta^5) = 0.
\end{array}
\end{equation*}
This is in agreement with the fixed point set sizes:
\begin{equation*}
\begin{array}{ccc}
| SYT((2,2,2))^{1}| = 5 & |SYT((2,2,2))^{j}| = 0 &
| SYT((2,2,2))^{j^2}| = 2 \\
| SYT((2,2,2))^{j^3}| = 3 & |SYT((2,2,2))^{j^4}| = 2
& |SYT((2,2,2))^{j^5}| = 0,
\end{array}
\end{equation*}
as predicted by Theorem 1.3.
\end{ex}

\begin{proof} (of Theorem 1.3)
By Corollary 3.6, $C$ does indeed act on $X$ by promotion.  Let $\zeta =
e^{\frac{2 \pi i}{n}}$.  Viewing $c_n = (1, 2, \dots, n)$ as a permutation
matrix in $S_n \subset GL_n(\mathbb{C})$, we get that $c$ is conjugate to
diag($1, \zeta, \zeta^2, \dots, \zeta^{n-1}$).  This means that for any $d \geq
0$ $c_n^d$ is conjugate to
diag($1, \zeta^d, \zeta^{2d}, \dots, \zeta^{d(n-1)}$).

Let $\chi^{\lambda}$ denote the character of the irreducible representation of
$S_n$ corresponding to $\lambda$.  It is well-known that $c_n$ is a regular
element of $S_n$, that is, there is an  eigenvector of the image of $c_n$ under
the reflection representation which avoids all reflecting hyperplanes.
A result of Springer (Proposition 4.5 of \cite{SpringerReg}) on regular
elements implies that, for any
$d \geq 0$, we have the character evaluation
$\chi^{\lambda}(c_n^d) =
\zeta^{d \kappa( \lambda )} f^{ \lambda}( \zeta^d)$,
where $\kappa(\lambda) = 0 \lambda_1 + 1 \lambda_2 + 2 \lambda_3 + \cdots$.
Since $\lambda = b^a$, we can compute that $\kappa( \lambda ) =
\frac{ba(a-1)}{2}$, which implies that $\zeta^{d \kappa( \lambda )} =
(-1)^{d(a-1)}$.

On the other hand by Proposition 3.5, $\chi^{\lambda}(c_n^d)$ is equal to the
trace of $(-1)^{d(a-1)} J^d$, where $J$ is the permutation matrix which records
promotion.  This trace is $(-1)^{d(a-1)}$ times the number of fixed points
$|X^{j^d}|$ of the action of $j^d$ on $X$.  The desired CSP follows.
\end{proof}

\section{A Construction of the Finite Dimensional Irreducible
$GL_k(\mathbb{C})$ Representations}

In the last section we proved a CSP regarding rectangular standard tableaux by
modeling the action of promotion by the image of the long cycle in $S_n$ under
the Kazhdan-Lusztig cellular representation of appropriate shape.  We want to
prove an analogue of this CSP for column strict tableaux of rectangular shape
having $n$ boxes and entries uniformly bounded by $k$.  In analogy with the
last section, we will construct a representation having dimension $CST(\lambda,
k)$ for rectangular $\lambda \vdash n$ under which some group element will act
as the promotion operator with respect to some basis.

It is well known that the irreducible finite dimensional representations of the
general linear group $GL_k(\mathbb{C})$ are paramaterized by partitions
$\lambda$ having at most $k$ rows and that the dimension of the irreducible
representation labeled by $\lambda$ is equal to $|CST(\lambda, k)|$.  The
symmetric group $S_k$ embeds naturally into $GL_k(\mathbb{C})$ as the subgroup
of permutation matrices and it will turn out that the long cycle
$(1,2,\dots,k)$ can be taken to model the action of promotion on column strict
tableaux provided our shape $\lambda$ is rectangular.  It remains, however, to
find a basis for our representation under which $(1,2,\dots,k)$ acts as a
monomial matrix corresponding to promotion.  This will involve a construction
of finite dimensional $GL_k(\mathbb{C})$ irreducible representations of
\emph{arbitrary} shape which we now present.  This construction is essentially
due to Du \cite{DuCBIrrRep} together with results of Skandera \cite{SkanNNDCB}
but we rederive it here for self-containment.  Our basis will essentially be a
homomorphic image of a subset of the dual canonical basis for the polynomial
ring $\mathbb{C}[x_{11}, \dots, x_{kk}]$ and, as such, will bear spiritual
similarities to Lusztig's canonical basis \cite{LusztigCB}.  As a starting
point of our construction, we introduce a family of polynomials called
immanants.  The homomorphic images of an appropriate subset of these
polynomials will be the zero weight space of our representations.

For a positive integer $n \in \mathbb{N}$, let
$x = (x_{ij})_{1 \leq i, j \leq n}$ be an $n \times n$ matrix of commuting
variables and let $\mathbb{C}[x_{11}, \dots, x_{nn}]$ be the complex polynomial
ring in these variables.  We will sometimes abbreviate the latter ring as
$\mathbb{C}[x_{ij}]_{1 \leq i,j \leq n}$.  Call a polynomial in
$\mathbb{C}[x_{ij}]_{1 \leq i,j \leq n}$ an \emph{immanant} if it belongs to
the $\mathbb{C}$-linear span of the permutation monomials
$\{ x_{1,w(1)} \cdots x_{n,w(n)} \,|\, w \in S_n \}$.  Thus, immanants form an
$n!$-dimensional complex vector space.  Given any polynomial $f(x_{11}, \dots,
x_{nn}) \in \mathbb{C}[x_{ij}]_{1 \leq i,j \leq n}$ and an $n \times n$ matrix
$A = (a_{ij})$ with entries in any commutative $\mathbb{C}-$algebra $R$, define
$f(A)$ to be the element
$f(a_{11}, \dots, a_{nn})$ of $R$ obtained by applying $f$ to $A$.

Following \cite{RSkanKLImm}, define for any $w \in S_n$ the
\emph{w-Kazhdan-Lusztig immanant} $\imm{w}(x)$ by the equation
\begin{equation}
\imm{w}(x) = \sum_{v \geq w} (-1)^{\ell(w,v)}
P_{w_o v, w_o w}(1) x_{1,v(1)} \cdots x_{n,v(n)}.
\end{equation}
Specializing to the identity permutation, we have that $\imm{1}(x) = \det(x)$.
So, at least superficially, KL immanants are deformations of the determinant.
It can be shown that the KL immanants share the properties of Schur
nonnegativity and total nonnegativity with the determinant
\cite{RSkanKLImm}, and that for the case of permutations $w$ which are
321-avoiding, the application of the KL immanant $\imm{w}(x)$ to the path
matrix of a planar network has a combinatorial interpretation which naturally
generalizes Lindstr\"om's Lemma \cite{RSkanTLImm}.
By the Bruhat triangularity of the KL polynomials, the set $\{ \imm{w}(x) \,|\,
w \in S_n \}$ forms a basis for the vector space of immanants.

We will find it necessary to work with KL immanants defined on variable sets
with repeated entries.
More precisely, given any pair of compositions $\alpha, \beta \models n$, we
define the matrix $x_{\alpha,\beta}$ to be
$(x_{\alpha(i),\beta(j)})_{1 \leq i,j \leq n}$.  Note that either or both of
$\ell(\alpha)$ or $\ell(\beta)$ may exceed $n$.  We also construct the
corresponding polynomial ring $\mathbb{C}[x_{\alpha(i),\beta(j)}]_{1 \leq i,j
\leq n}$.  For a permutation $w \in S_n$, denote by $\imm{w}(x_{\alpha,
\beta})$ the element of $\mathbb{C}[x_{\alpha(i),\beta(j)}]_{1 \leq i,j \leq
n}$ which results in applying the $w$-KL immanant to the matrix
$x_{\alpha,\beta}$.  So, for example, we have that
$\imm{w}(x) = \imm{w}(x_{1^n,1^n})$.  In this paper we will mostly be
interested in the case where $\beta = 1^n$.

\begin{ex}
Take $n = 3$ and let $\alpha = (1, 2)$ and $\beta = (1,1,1)$.  Our matrix
$x_{\alpha,\beta}$ is given by
\begin{equation*}
x_{(1,2),(1,1,1)} =
\left(
\begin{array}{ccc}
    x_{11} & x_{12} & x_{13} \\
    x_{11} & x_{12} & x_{13} \\
    x_{21} & x_{22} & x_{23}.
\end{array}
\right)
\end{equation*}
Let $w = 213 \in S_3$.  Recalling that every KL polynomial for pairs of
permutations in $S_3$ is either $0$ or $1$ depending on whether the pairs are
Bruhat comparable, we see that
\begin{equation*}
\imm{213}(x_{(1,2),(1,1,1)}) = x_{12}x_{11}x_{23} - x_{12}x_{13}x_{21} -
x_{13}x_{11}x_{22} + x_{13}x_{12}x_{21}.
\end{equation*}
On the other hand, computing the KL immanant corresponding to the permutation
$231 \in S_3$ yields the result
\begin{equation*}
\imm{231}(x_{(1,2),(1,1,1)}) = x_{12}x_{13}x_{21} - x_{13}x_{12}x_{21} = 0.
\end{equation*}
\end{ex}

As the above example shows, it can happen that a polynomial of the form
$\imm{w}(x_{\alpha,\beta})$ is equal to zero.  Indeed, since $\imm{1}(x) =
\det(x)$, such a polynomial is nonzero for $w = 1$ if and only if all of the
parts of $\alpha$ and $\beta$ are at most $1$.  It is possible to derive a
criterion based on the RSK correspondence for determining precisely when the
polynomials $\imm{w}(x_{\alpha,\beta})$ vanish \cite{SkanNNDCB}.
A remarkable fact, due to Skandera, is that the nonvanishing polynomials of the
form $\imm{w}(x_{\alpha,\beta})$ form a \emph{basis} for the polynomial ring
$\mathbb{C}[x_{11}, \dots, x_{kk}]$ \cite{SkanNNDCB}.  Moreover,
Skandera showed that the basis so constructed is essentially equivalent to
Lusztig's dual canonical basis.  We will not need the full force of this
result, and now state the theorem which will be relevant for our purposes.

\begin{thm}$($Skandera \cite{SkanNNDCB}$)$  Let $k \geq 0$.  The nonzero
elements of the set
$\{ \imm{w}(x_{\alpha,\beta}) \}$, where $w$ ranges over $S_n$ and $\alpha$ and
$\beta$ range over all possible compositions of $n$ having length $k$, are
linearly independent and a subset of the dual canonical basis of the polynomial
ring $\mathbb{C}[x_{11}, \dots, x_{kk}]$ in $k^2$ variables.
\end{thm}

We now relate polynomial rings to representation theory.
Let $Y = \mathbb{C}^k$ and $Z = \mathbb{C}^n$ be two complex vector spaces of
dimensions $k$ and $n$, respectively.  Let $Y^{*}$ and $Z^{*}$ denote their
dual spaces with standard bases
$\{ y_1, \dots, y_k \}$ and $\{ z_1, \dots, z_n \}$, respectively.  Now the
tensor product $Y^{*} \otimes Z^{*}$ has basis $x_{ij} := y_i \otimes z_j$, for
$1 \leq i \leq k$ and $1 \leq j \leq n$.  In this way, we identify the
symmetric algebra Sym($Y^{*} \otimes Z^{*}$) with the polynomial ring
$\mathbb{C}[x_{ij}]_{1 \leq i \leq k, 1 \leq j \leq n}$.  This space carries an
action of the general linear group $GL(Y) = GL_k(\mathbb{C})$, where matrices
act on the first component of simple tensors by $g \cdot (f \otimes h) := (f
g^{-1}) \otimes h$.
Taking $n = k$ and $i \in [n-1]$, viewing the adjacent transposition $s_i$ as
an element of $S_n \subset GL_k(\mathbb{C})$, we quote a result from
\cite{RSkanKLImm} about this action.

\begin{lem}
Let $w \in S_n$.  We have that
\begin{equation}
s_i \imm{w}(x) =
\begin{cases}
-\imm{w}(x) & s_i w > w,\\
\imm{w}(x) + \imm{s_iw}(x) +
\sum_{s_iz > z} \mu(w,z) \imm{z}(x) & s_i w < w.
\end{cases}
\end{equation}
\end{lem}

Using the formula in the above lemma, one can show that the KL immanants form a
cellular basis for the vector space of immanants.  More precisely, identifying
permutations with their images under RSK, for any $\lambda \vdash n$ and $T \in
SYT(\lambda)$, we have that the space
\begin{equation*}
V'_{T,n,1^n} := \mathbb{C} \{ \imm{(T,P)}(x) \,|\, P \in SYT(\lambda) \}
\oplus
\bigoplus_{\nu >_{dom} \lambda} \mathbb{C}
\{\imm{(U,S)}(x) \,|\, U, S \in SYT(\nu) \}
\end{equation*}
is closed under the left action of $S_n$.  This closure is essentially a
consequence of the fact that the formula in Lemma 4.2 is almost the same as the
formula for the action of $s_i$ on the KL basis of $\csn$.

Since $V'_{T,n,1^n}$ is closed under the action of $\csn$ for \emph{any} $T$,
it follows immediately that the quotient space
\begin{equation*}
V_{T,n,1^n} := V'_{T,n,1^n} / (\bigoplus_{\nu >_{dom} \lambda} \mathbb{C}
\{\imm{(U,S)}(x) \,|\, U, S \in SYT(\nu) \})
\end{equation*}
carries the irreducible $S_n$-representation corresponding to the shape
$\lambda$.  A basis for $V_{T,n,1^n}$ is given by the image of the set $\{
\imm{(T,P)}(x) \,|\, P \in SYT(\lambda) \}$ under the canonical projection map.
Letting $I_{1^n}(P)$ denote the image of $\imm{(T,P)}(x)$ under this
projection, we can write this basis as $\{I_{1^n}(P) \,|\, P \in
SYT(\lambda)\}$.  The $1^n$ appearing in this notation will be fully justified
when we broaden our scope to nonstandard tableaux.  By a change of label
argument, the representation of $S_n$ on the quotient space $V_{T,n,1^n}$ does
not depend on the choice of the tableau $T$, and is given by
\begin{equation}
s_i I_{1^n}(P) =
\begin{cases}
-I_{1^n}(P)  & \text{if $i \in D(P)$}\\
I_{1^n}(P) + \sum_{i \in D(Q)} \mu[P,Q] I_{1^n}(Q) & \text{if $i \notin D(P)$}.
\end{cases}
\end{equation}

So far we have used KL immanants to construct representations of the group
algebra $\csn$.  The goal of this section is to construct modules over
$GL_k(\mathbb{C})$.  To do this, we first recall combinatorial notions of
standardization and semistandardization on tableaux.  These operations give us
a means to transform row strict tableaux into standard tableaux and vice versa,
when possible.

It is always possible to transform a row strict tableau into a standard
tableau.
Given a partition $\lambda \vdash n$ and a row strict tableau $P \in
RST(\lambda, k, \alpha)$ for some $\alpha \models n$, define the
\emph{standardization} $std(P)$ of $P$ to be the element of $SYT(\lambda)$
given by replacing the $\alpha_1$ $1s$ in $P$ with the numbers $[1, \alpha_1]$
increasing down columns, replacing the $\alpha_2$ $2s$ in $P$ with the numbers
$[\alpha_1+1,\alpha_1 + \alpha_2]$ increasing down columns, and so on.

\begin{ex}
Suppose that $P$ is the tableau given by
\begin{equation*}
P =
\begin{array}{ccc}
1 & 7 & 8 \\
1 & 9 &   \\
2 & 9 &
\end{array} \in RST((3,2,2), 9, (2,1,0,0,0,0,1,1,2)),
\end{equation*}
then we have that
\begin{equation*}
std(P) =
\begin{array}{ccc}
1 & 4 & 5 \\
2 & 6 &   \\
3 & 7 &
\end{array} \in SYT((3,2,2)).
\end{equation*}
\end{ex}

The reverse case is more complicated.  Given $\lambda \vdash n$ and a standard
tableau $T \in SYT(\lambda)$ along with a composition $\alpha \models n$, say
that $T$ is \emph{$\alpha$-semistandardizable} if the sequences of numbers $[1,
\alpha_1]$,
$[\alpha_1+1, \alpha_1 + \alpha_2]$, $\dots$, all occur in vertical strips in
$T$ which increase down columns.
Equivalently, the standard tableau $T$ is $\alpha$-semistandardizable if and
only if $D(T)$ contains the union of the intervals
$[1, \alpha_1)$, $[\alpha_1 + 1, \alpha_2)$, $\dots$.
Define the \emph{$\alpha$-semistandardization} $rst_{\alpha}(T)$ of $T$
$\alpha$-semistandardizable to be the element of
$RST(\lambda, k, \alpha)$ formed by replacing the numbers in $[1, \alpha_1]$ in
$T$ by $1s$, the numbers in $[\alpha_1+1, \alpha_1 + \alpha_2]$ in $T$ by $2s$,
and so on.

\begin{ex}
The tableau $T \in SYT((3,2,2))$ shown below is not
$(1,2,3,1)$-semistandardizable because the $3$ occurs in a higher row than the
$2$ and therefore $rst_{(1,2,3,1)}(T)$ is undefined.  Notice, however, that the
tableau obtained by replacing the $3$ in $T$ with a $2$, the $4, 5,$ and $6$
with $3$s, and the $7$ with a $4$ is in fact row strict.
\begin{equation*}
T =
\begin{array}{ccc}
1 & 3 & 4 \\
2 & 5 &   \\
6 & 7 &   \\
\end{array}
\end{equation*}
On the other hand, the tableau $U \in SYT((3,2,2))$ shown is
$(1,2,3,1)$-semistandardizable and its $(1,2,3,1)$-semistandardization is
shown.
\begin{equation*}
U =
\begin{array}{ccc}
1 & 2 & 4 \\
3 & 5 &   \\
6 & 7 &   \\
\end{array}
rst_{(1,2,3,1)}(U) =
\begin{array}{ccc}
1 & 2 & 3 \\
2 & 3 &   \\
3 & 4 &   \\
\end{array}
\end{equation*}
\end{ex}

Our terminology here may be slightly misleading since semistandard tableaux are
usually defined in the literature to be column strict rather than row strict.
This difference is cosmetic, though, and this notion of semistandardization
will allow for cleaner statements of our results.

The following is immediate.

\begin{lem}
Let $T \in SYT(\lambda)$ be $\alpha$-semistandardizable.  We have that
$std(rst_{\alpha}(T)) = T$.  Moreover, if $U \in RST(\lambda, k, \alpha)$, then
$std(U)$ is $\alpha$-standardizable.
\end{lem}

Therefore, for any composition $\alpha \models n$, standardization injects
$RST(\lambda, k, \alpha)$ into $SYT(\lambda)$ and $\alpha$-semistandardization
gives a bijection between the $\alpha$-semistandardizable elements of
$SYT(\lambda)$ and $RST(\lambda, k, \alpha)$.
The combinatorial notion of semistandardizable tableaux can be related to the
vanishing of KL immanants as follows.  This will be a key point in our
construction of modules over the general linear group.

\begin{lem}
Let $U,T \in SYT(\lambda)$ and let $k \in \mathbb{N}$.  We have that
$\imm{(T,U)}(x_{\alpha,1^n}) = 0$ if and only if $U$ is not
$\alpha$-semistandardizable.  Moreover, the set $\{
\imm{(T',U')}(x_{\alpha',1^n}) \}$ ranging over all possible $\alpha' \models
n$ with $\ell(\alpha') \leq k$, $U', T' \in SYT(\lambda)$, and $U'$ that are
$\alpha'$-semistandardizable, is linearly independent.
\end{lem}

\begin{proof}
Let $S_\alpha$ denote the parabolic subgroup of $S_n$ which setwise fixes the
intervals $[1, \alpha_1]$, $[\alpha_1+1, \alpha_1 + \alpha_2]$, and so on.  By
Skandera's results on the dual canonical basis \cite{SkanNNDCB}, we have that
$\imm{(T,U)}(x_{\alpha,1^n}) \neq 0$ if and only if the inverse image $w$ of
$(T,U)$ under RSK is a Bruhat maximal element of some left coset of $S_\alpha$
in $S_n$.  It is easy to show using the properties of jeu-de-taquin and RSK
that this happens if and only if $U$ is $\alpha$-semistandardizable.  The claim
about linear independence follows from Theorem 4.1.
\end{proof}

\begin{ex}
As we have seen in Example 4.1, we have that $\imm{213}(x_{(2,1),(1,1,1)})$ is
nonzero and
$\imm{231}(x_{(2,1),(1,1,1)}) = 0$.  It is easy to see that the permutation 213
and 231 row insert as follows:
\begin{equation*}
213 \mapsto \left(
\begin{array}{ccccc}
1 & 3 &  & 1 & 3 \\
2 &   &, & 2 &
\end{array} \right)
\end{equation*}
\begin{equation*}
231 \mapsto \left(
\begin{array}{ccccc}
  1 & 3 &  & 1 & 2 \\
  2 &   & ,& 3 &
\end{array}
\right).
\end{equation*}
Therefore, the recording tableau for $213$ is $(2,1)$-semistandardizable
whereas the recording tableau for $231$ is not.  The nonvanishing and vanishing
of the associated polynomials is therefore predicted by Lemma 4.4.
\end{ex}

We are finally ready to define our general linear group modules.  For any $T
\in SYT(\lambda)$, define $V'_{T,k}$ to be the space
\begin{equation*}
V'_{T,k} :=
\bigoplus_{\alpha} \bigoplus_{U} \mathbb{C}\{\imm{(T,U)}(x_{\alpha,1^n})
\oplus
\bigoplus_{\nu <_{dom} \lambda} \bigoplus_{\alpha'} \bigoplus_{P,Q}
\mathbb{C}\{\imm{(P,Q)}(x_{\alpha',1^n})\},
\end{equation*}
where the first set ranges over all compositions $\alpha \models n$ such that
$\ell(\alpha) = k$ and all $U \in SYT(\lambda)$ which are
$\alpha$-semistandardizable and the second set ranges over all compositions
$\alpha' \models n$ with $\ell(\alpha') = k$ and all pairs of tableaux $P, Q
\in SYT(\nu)$.  We first show that $V'_{T,k}$ is closed under the action of
$GL_k(\mathbb{C})$.

\begin{lem}
$V'_{T,k}$ is a left $GL_k(\mathbb{C})$-module.
\end{lem}

\begin{proof}
For $1 \leq i < k$ and $z \in \mathbb{C}$, let $E_{i,z}$ (resp. $F_{i,z}$)
denote the elementary matrices in $GL_k(\mathbb{C})$ which are obtained by
replacing the 0 in position ($i$, $i+1$) (resp. ($i+1$, $i$)) with a $z$.
We show that $V'_{T,k}$ is closed under the action of the permutation matrices
$S_k \subset GL_k(\mathbb{C})$, the Cartan subgroup $H$ of diagonal matrices in
$GL_k(\mathbb{C})$, and $E_{i,z}$ and $F_{i,z}$ for $1 \leq i < k$ and $z \in
\mathbb{C}$.  Since $GL_k(\mathbb{C})$ is generated by these matrices, the
result will follow.

To show that $V'_{T,k}$ is closed under the action of $S_k$, let $1 \leq i <
k$.  We show that $V'_{T,k}$ is closed under the action of $s_i \in S_k$.
Let $\imm{(P,Q)}(x_{\beta, 1^n})$ be a basis element of $V'_{T,k}$.  If we let
$s_i \cdot \beta$ be the composition given by
$s_i \cdot \beta := (\beta_1, \dots, \beta_{i-1}, \beta_{i+1}, \beta_i,
\beta_{i+2}, \dots, \beta_k) \models n$, it follows that the polynomial
$s_i \cdot \imm{(P,Q)}(x_{\beta,1^n})$
is equal to the image of
$w \cdot \imm{(P,Q)}(x_{1^n,1^n})$
under the $\mathbb{C}$-algebra homomorphism
$x_{i,j} \mapsto x_{\beta'(i),j}$.  Here $w$ is the permutation in $S_n$ which
is obtained by replacing the square diagonal submatrix of size $|\beta_i +
\beta_{i+1}|$ in the $n \times n$ identity matrix corresponding to rows and
columns in
$(\beta_1 + \cdots + \beta_{i-1}, \beta_1 + \cdots + \beta_{i+1}]$ with an
antidiagonal matrix of 1's.  It follows from Theorem 2.7 and the analogous
actions of Equations 2.10 and 4.4 that this latter polynomial remains in
$V'_{T,k}$, as desired.

Given any diagonal matrix $A =$ diag($a_1, \dots, a_k$) in $H$, it is easily
seen that the immanant $\imm{(P,Q)}(x_{\beta,1^n})$ is an eigenvector for the
operator $A$ with eigenvalue
$a^{\beta} := a_1^{\beta(1)} \cdots a_k^{\beta(n)}$.  Thus, $V'_{T,k}$ is
closed under the action $H$.

Finally, we show that $V'_{T,k}$ is closed under the action of the $E_{i,z}$
and $F_{i,z}$.  To simplify notation, we show this only for the case of nonzero
$z$ and $F_{i, z^{-1}}$.  The other cases are similar.  Let $\beta = (\beta_1,
\dots, \beta_k)$ be a composition of $n$ of length $k$ and let $w$ be a
permutation in $S_n$ with $w \mapsto (P,Q)$.  Let $\nu$ be the shape of $P$ and
$Q$.
Let $I := \{ \ell \in [n] \,|\, \beta(\ell) = i \}$.
The image of $\imm{(P,Q)}(x_{\beta,1^n})$ under the action of $F_{i,z^{-1}}$ is
the polynomial
\begin{equation*}
\sum_{v \in S_n} (-1)^{\ell(v,w)} P_{w_o v, w_o w}(1)
\prod_{\ell \in I} (x_{i v(\ell)} + z x_{(i+1) v(\ell)}) \prod_{m \in [n] - I}
x_{\beta(m) v(m)}.
\end{equation*}
Expanding out the terms in parenthesis and regrouping, we see that the above
expression is equal to
\begin{equation*}
\sum_{J \subseteq I} z^{|J|}
\sum_{v \in S_n} (-1)^{\ell(v,w)} P_{w_o v, w_o w}(1)
\prod_{j \in J} x_{(i+1) v(j)}
\prod_{\ell \in I - J} x_{i v(\ell)}
\prod_{m \in [n] - I} x_{\beta(m), v(m)}.
\end{equation*}
Fix a subset $J \subseteq I$.  Let $y$ be any permutation in $S_n$ which fixes
every letter not in $I$ and rearranges the letters of $I$ so that the letters
in $J$ are mapped into a contiguous suffix.  Let $\gamma \models n$ be the
composition of $n$ defined by
\begin{equation*}
\gamma := (\beta_1, \dots, \beta_{i-1}, \beta_i - |J|, \beta_{i+1} + |J|,
\beta_{i+1}, \dots, \beta_k).
\end{equation*}
By the discussion following Lemma 4.2, the image of
\begin{equation*}
\sum_{v \in S_n} (-1)^{\ell(v,w)} P_{w_o v, w_o w}(1)
\prod_{j \in J} x_{(i+1) v(j)}
\prod_{\ell \in I-J} x_{i v(\ell)}
\prod_{m \in [n]-I} x_{\beta(m), v(m)}
\end{equation*}
under the action of $y^{-1}$ is a complex linear combination of terms of the
form
\linebreak$\imm{(P',Q')}(x_{\gamma,1^n})$, where either $P'$ and $Q'$ are both
standard tableaux of shape $\nu$ and $P = P'$ or $P'$ and $Q'$ are both
standard tableaux of shape strictly dominating $\nu$.  Since the subset $J$ was
arbitrary and since we already know that $V'_{T,k}$ is closed under the action
of $S_k$, it follows that $V'_{T,k}$ is stable under the action of
$F_{i,z^{-1}}$, as desired.
\end{proof}

The above result can be found in \cite{DuCBIrrRep} together with
\cite{SkanNNDCB}.  At any rate, we define $V_{T,k}$ to be the quotient
$GL_k(\mathbb{C})$-module given by
\begin{equation*}
V_{T,k} := V'_{T,k} / \left( \bigoplus_{\nu >_{dom} \lambda}
\bigoplus_{\alpha'} \bigoplus_{P,Q}
\mathbb{C}\{\imm{(P,Q)}(x_{\alpha',1^n})\} \right),
\end{equation*}
\noindent
where the second direct sum ranges over compositions $\alpha' \models n$ with
$\ell(\alpha') = k$ and the third direct sum ranges over all standard tableaux
$P$ and $Q$ of shape $\nu$.
Basis elements of $V_{T,k}$ are given by the images of the polynomials
$\imm{(T,U)}(x_{\alpha,1^n})$ for compositions $\alpha = (\alpha_1, \dots,
\alpha_k) \models n$ and $\alpha$-semistandardizable $U$.  The image of the
above polynomial in $V_{T,k}$ shall be abbreviated $I_{\alpha}(U')$, where $U'$
is the unique element of $RST(\lambda, k, \alpha)$ such that $std(U') = U$.

\begin{thm}
The $GL_k(\mathbb{C})$-module $V_{T,k}$ is isomorphic to the dual of the
irreducible finite dimensional $GL_k(\mathbb{C})$-module with highest weight
$\lambda'$, where $\lambda$ is the shape of $T$.  Moreover, for any $\alpha
\models n$ with $\ell(\alpha) = k$, the weight space of $V_{T,k}$ corresponding
to $-\alpha$ is equal to the $\mathbb{C}$-linear span of $\{ I_{\alpha}(U) \}$,
where $U$ ranges over $RST(\lambda, k, \alpha)$.
\end{thm}

\begin{proof}
We compute the Weyl character of $V_{T,k}$.  Let $h :=$ diag$(a_1,\dots,a_k)$
be an element of the Cartan subgroup of $GL_k(\mathbb{C})$ for some nonzero
complex numbers $a_1, \dots, a_k$.  It is easy to see that the action of $h$ on
some basis element $I_{\alpha}(U)$ for $U$ row strict with content $\alpha$ is
given by
\begin{equation*}
h \cdot I_{\alpha}(U) = a_1^{-\alpha_1} \cdots a_k^{-\alpha_k} I_{\alpha}(U).
\end{equation*}
(Here we recall that our action of $GL_k(\mathbb{C})$ came from the
\emph{contragredient} action.)
Since the set of all $I_{\alpha}(U)$ where $U$ ranges over all row strict
tableaux of shape $\lambda$ and entries $\leq k$ forms a basis for $V_{T,k}$,
it follows that $h$ acts on $V_{T,k}$ with trace
\begin{equation*}
\sum_{\alpha \models n, \ell(\alpha) = k} \sum_{U \in RST(\lambda, k, \alpha)}
a_1^{-\alpha_1} \cdots a_{k}^{-\alpha_k}.
\end{equation*}
This latter sum is the combinatorial definition of the Schur function
$s_{\lambda'}(a_1^{-1}, \dots, a_k^{-1})$, which is the Weyl character of the
dual of the irreducible finite dimensional $GL_k(\mathbb{C})$-module with
highest weight $\lambda'$.

The claim about the weight space decomposition of $V_{T,k}$ is obvious.
\end{proof}

\begin{ex}
We illustrate our construction for the case $k = 3$, $n = 4$, and $\lambda =
(2,2)$.  Let $T$ be the standard tableau of shape $(2,2)$ given by
\begin{equation*}
T =
\begin{array}{cc}
1 & 2 \\
3 & 4.
\end{array}
\end{equation*}
The space $V'_{T,3}$ is equal to the complex span of
\begin{align*}
| RST((2,2),3)| +& |RST((2,1,1),3)|*|SYT(2,1,1)| + \\
 &|RST((1,1,1,1),3)|*|SYT(1,1,1,1)| = 6 + 15*3 + 35*1  = 86
\end{align*}
linearly independent polynomials which are obtained as above by applying KL
immanants corresponding to permutations in $S_4$ which row insert to any of the
shapes $(2,2)$, $(2,1,1)$, or $(1,1,1,1)$ to $4 \times 4$ matrices in the $9$
variables
$x_{11}, \dots, x_{33}$ with possibly repeated rows.  Also as explained above,
we only consider KL immanants corresponding to the permutations with shape
$(2,2)$ whose insertion tableaux are equal to $T$.  By Lemma 4.5, the 86
dimensional space $V'_{T,3}$ is closed under the action of $GL_3(\mathbb{C})$.

The space $V_{T,3}$ is obtained from $V'_{T,3}$ by modding out by the complex
span of all polynomials obtained by the application of KL immanants whose
associated permutations do not have shape $(2,2)$ under RSK.  The space
$V_{T,3}$ has dimension $|RST((2,2),3)| = 6$.  A complex basis of $V_{T,3}$ is
the set $\{ I_{\alpha}(U) \}$ as defined above, where $U$ is a row strict
tableau of shape $(2,2)$ and content composition $\alpha$ satisfying
$\ell(\alpha) = 3$.  We write down an element of this basis explicitly.

First note that the permutations in $S_4$ with insertion tableau $T$ are
precisely 3412 and 3142.  It can be shown that the KL immanants corresponding
to these permutations are
\begin{equation*}
\imm{3412}(x) = x_{13}x_{24}x_{31}x_{42} - x_{14}x_{23}x_{31}x_{42} -
x_{13}x_{24}x_{32}x_{41} + x_{14}x_{23}x_{32}x_{41}
\end{equation*}
and

\begin{align*}
\imm{3142}(x) &= x_{13}x_{21}x_{34}x_{42} - x_{13}x_{22}x_{34}x_{41} -
x_{13}x_{24}x_{31}x_{42} - x_{14}x_{21}x_{33}x_{42}\\ &
+ x_{14}x_{22}x_{33}x_{41} + x_{13}x_{24}x_{32}x_{41} +
x_{14}x_{23}x_{31}x_{42} - x_{14}x_{23}x_{32}x_{41}.
\end{align*}

Let $T_o$ be the tableau given by
\begin{equation*}
T_o =
\begin{array}{cc}
1 & 3 \\
2 & 4
\end{array}
\end{equation*}
The set $RST((2,2),3)$ is equal to
\begin{equation*}
\{ rst_{(1,2,1)}(T) , rst_{(2,2,0)}(T_o), rst_{(2,0,2)}(T_o),
rst_{(0,2,2)}(T_o), rst_{(1,1,2)}(T_o), rst_{(2,1,1)}(T_o) \}.
\end{equation*}
To find the basis element $I_{(1,2,1)}( rst_{(1,2,1)}(T) )$, we apply the
3412-KL immanant to the matrix $x_{(1,2,1),(1,1,1,1)}$.  This results in
\begin{equation*}
\imm{3412}(x_{(1,2,1),(1,1,1,1)}) = x_{13} x_{24} x_{21} x_{32} - x_{14} x_{23}
x_{21} x_{32} - x_{13} x_{24} x_{22} x_{31} + x_{14} x_{23} x_{22} x_{31}
\end{equation*}
and $I_{(1,2,1)}( rst_{(1,2,1)}(T) )$ is the homomorphic image of this in the
quotient $V'_{T,3}$.
\end{ex}

\section{Promotion on Column Strict Tableaux}

The goal in this section is to prove a CSP for column strict tableau using the
description of irreducible $GL_k(\mathbb{C})$-modules given in the last
section.  This will essentially involve showing that the long cycle
$(1,2,\dots,k) \in S_k \subset GL_k(\mathbb{C})$ acts as a monomial matrix
corresponding to promotion in the \emph{rectangular} irreducible
representations constructed in Section 4.

We start by defining a collection of epimorphisms which will allow us to work
with our representations more easily.
Given $\alpha \models n$, we define an epimorphism $\mathbb{C}$-algebras
$\pi_{\alpha}: \mathbb{C}[x_{ij}]_{1 \leq i,j \leq n} \rightarrow
\mathbb{C}[x_{\alpha(i)j}]_{1 \leq i,j \leq n}$ by the formula
$\pi_{\mu}(x_{ij}) = x_{\alpha(i)j}$.  The relevant vanishing properties of KL
immanants may be stated cleanly in terms of the maps $\pi_{\alpha}$.

\begin{lem}
Let $U, T \in SYT(\lambda)$.  We have that
$\pi_{\alpha}(\imm{(U,T)}(x))$ is nonzero if and only if $U$ is
$\alpha$-semistandardizable, in which case
$\pi_{\alpha}(\imm{(U,T)}(x)) = \imm{(U,T)}(x_{\alpha,1^n})$.
\end{lem}

\begin{proof}
This is straightforward from the definition of the KL immanants and
$\pi_{\alpha}$ together with Lemma 4.4.
\end{proof}

We record the action of the long cycle in $S_n$ on immanants (which correspond
to standard tableaux).  This will be related to the action of the long cycle in
$S_k$ on semistandard tableaux via the maps $\pi_{\alpha}$.

\begin{lem}
Let $\lambda = b^a$ be a rectangle and let $P \in SYT(\lambda)$.  Let $c_n =
(1, 2, \dots n) \in S_n$ be the long cycle.  We have that $c_n \cdot I_{1^n}(P)
= (-1)^{b-1} I_{1^n}(j(P))$.
\end{lem}
\begin{proof}
This is straightforward using the corresponding result for the KL basis of
$\csn$ (Proposition 3.5) and the action of $S_n$ on immanants given in Lemma
4.2.
\end{proof}

We also record how the combinatorial operations of promotion and
standardization commute.

\begin{lem}
Let $\lambda \vdash n$ be a rectangle and $\alpha = (\alpha_1, \dots, \alpha_k)
\models n$.  We have the following equality of operators on $SYT(\lambda)$:
\begin{equation}
j \circ rst_{\alpha} = rst_{c_k \cdot \alpha} \circ j^{\alpha_k},
\end{equation}
where $c_k \cdot \alpha$ is the composition of $n$ of length $k$ given by
$c_k \cdot \alpha := (\alpha_k, \alpha_1, \cdots, \alpha_{k-1})$ and the right
hand side is defined if and only if the left hand side is defined.
\end{lem}

\begin{proof}
Observe that for any tableau $P \in SYT(\lambda)$, $i+1$ occurs strictly south
and weakly west of $i$ in $P$ if and only if $i \in D(P)$.  For $i < n$ this is
equivalent to the condition $i \in D_e(P)$.  Therefore, this follows from the
definition of jeu-de-taquin as well as the cyclic action of jeu-de-taquin on
the extended descent set of rectangular tableaux from Lemma 3.3.
\end{proof}

The following lemma shows how the $\pi_{\alpha}$ allow the transfer of
information from the standard to the semistandard case.

\begin{lem}
Let $c_n = (1, 2, \dots, n)$ be the long cycle in $S_n$ and let
$c_k = (1, 2, \dots, k)$ be the long cycle in $S_k$.  For $\alpha = (\alpha_1,
\dots, \alpha_k) \models n$ we have a left action of $c_n$ on
$\mathbb{C}[x_{ij}]_{1 \leq i,j \leq n}$ and  $d$ maps
$\mathbb{C}[x_{\alpha(i)j}]_{1 \leq i,j \leq n}$ into $\mathbb{C}[x_{c_k \cdot
\alpha(i)j}]_{1 \leq i,j \leq n}$.  We have the following commutative square.
\begin{equation}
\begin{array}{ccccc}
& & c_n^{\alpha_k} & & \\
& \mathbb{C}[x_{ij}]_{1 \leq i,j \leq n}  & \longrightarrow &
\mathbb{C}[x_{ij}]_{1 \leq i,j \leq n} & \\
 \pi_{\alpha} & \downarrow & & \downarrow &
\pi_{c_k \cdot \alpha}   \\
& \mathbb{C}[x_{\alpha(i)j}]_{1 \leq i,j \leq n} &
\longrightarrow & \mathbb{C}[x_{c_k \cdot \alpha(i)j}]_{1 \leq i,j \leq n} &\\
& & c_k & &
\end{array}
\end{equation}
\end{lem}
\begin{proof}
The commutativity of this diagram is easily checked on the generators $x_{ij}$
of the algebra in the upper left.
\end{proof}

Recall that in the last section we constructed representations $V_{T,k}$ of
$GL_{k}(\mathbb{C})$ for a fixed standard tableau $T$ with $n$ boxes.  Let
$V_{T,k,\alpha}$ be the subspace of $V_{T,k}$ generated by the elements
$\{I_{\alpha}(U)\}$, where $U$ ranges over $RST(\lambda,k,\alpha)$.  By Theorem
4.6, the spaces $V_{T,k,\alpha}$ give the weight space decomposition of the
irreducible representation $V_{T,k}$.
In terms of this weight space decomposition, the commutative square (5.2)
implies the following commutative square:

\begin{equation}
\begin{array}{ccccc}
& & c_n^{\alpha_k} & & \\
& V_{T, n, 1^n}  & \longrightarrow &
V_{T, n, 1^n} & \\
\pi_{\alpha} & \downarrow & & \downarrow &
\pi_{c_k \cdot \alpha} \\
& V_{T,k,\alpha}  &
\longrightarrow &
V_{T,k,c_k \cdot \alpha} & \\
& & c_k & &
\end{array}.
\end{equation}
Using this square, we can relate the action of $c_k$ on $V_{T,k}$ to promotion.

\begin{prop}
Given $U \in RST(\lambda, k, \alpha)$ for $\lambda = b^a$ rectangular and
$\alpha = (\alpha_1, \dots, \alpha_k) \models n$,
we have that
\begin{equation}
c_k \cdot I_{\alpha}(U) = (-1)^{\alpha_k(b-1)}
I_{c_k \cdot \alpha}(j(U)).
\end{equation}
\end{prop}
\begin{proof}
We have the following chain of equalities:
\begin{align*}
c_k I_{\alpha}(U) &= \pi_{c_k \cdot \alpha}
c_n^{\alpha_k} I_{1^n}(std(U))\\
&= \pi_{c_k \cdot \alpha}(
(-1)^{\alpha_k(b-1)} I_{c_k \cdot \alpha}
(rst_{c_k \cdot \alpha} \circ j^{\alpha_k} \circ std(U))\\
&= (-1)^{\alpha_k(b-1)} I_{c_k \cdots \alpha}(j(U)).
\end{align*}
The first equality comes from the commutative square (5.3) and Lemma 5.2.  The
second comes from Lemmas 5.2 and 5.3.  The third is again Lemma 5.3 and the
definition of $\pi_{c_k \cdot \alpha}$.
\end{proof}

\begin{cor}
For $\lambda$ rectangular, the order of $j$ on $RST(\lambda, k)$ (or
$CST(\lambda,k)$) is equal to $k$ unless $\lambda$ consists of a single column (respectively, single row)
and $k = |\lambda|$.
\end{cor}
\begin{proof}
By Proposition 5.5, every element of $RST(\lambda,k)$ is fixed by the operator
$j^k$.  If $\lambda$ is any rectangular shape other than a column (provided
$RST(\lambda,k)$ is nonempty), it's easy to produce a row strict tableau of
content
composition having no nontrivial cyclic symmetry.
\end{proof}

\begin{ex}
If $\lambda = (2,2)$ and $k = 3$, the set $CST((2,2),3)$ contains $6$ elements
and promotion acts as the following permutation:

\begin{equation*}
\left(
\begin{array}{cccccccc}
1 & 1 &   & 2 & 2 &   & 1 & 1 \\
2 & 2 & , & 3 & 3 & , & 3 & 3
\end{array}
\right)
\left(
\begin{array}{cccccccc}
1 & 2 &   & 1 & 2 &   & 1 & 1 \\
2 & 3 & , & 3 & 3 & , & 2 & 3
\end{array}
\right),
\end{equation*}
which does indeed have order $k = 3$.
\end{ex}

Finally we are able to prove Theorem 1.4.
As in the standard 
tableau case of Theorem 1.3, the hypothesis that $\lambda$ is rectangular is
necessary.
For arbitrary shapes $\lambda$, the order of the operator $j$ on $CST(\lambda,
k)$ is unknown.  Also as with the standard case, even if we removed the
rectangular condition and  replaced $C$ by the cyclic group of size the same as
the order of $j$ on $CST(\lambda, k)$ (still taking $X(q)$ to be the principal
specialization of $s_\lambda$), this result would not be true.

\begin{ex}
Keeping with the earlier example of $\lambda = (2,2)$ and $k = 3$, we compute
that $\kappa((2,2)) = 2$ and
$s_{(2,2)}(x_1, x_2, x_3) =
x_1^2 x_2^2 + x_2^2 x_3^2 + x_1^2 x_3^2 +
x_1 x_2^2 x_3 + x_1 x_2 x_3^2 + x_1^2 x_2 x_3$.  Maintaining the notation of
Theorem 5.7, we therefore have that
$X(q) = 1 + q + 2q^2 + q^3 + q^4$.  Letting
$\zeta = e^{\frac{2 \pi i}{3}}$, we see that
\begin{equation*}
\begin{array}{ccc}
X(1) = 6 & X(\zeta) = 0 & X(\zeta^2) = 0.
\end{array}
\end{equation*}
These numbers agree with the fixed point set sizes:
\begin{equation*}
\begin{array}{ccc}
| X^{1}| = 6 &
| X^{j}| = 0 &
| X^{j^2}| = 0,
\end{array}
\end{equation*}
as predicted by Theorem 1.4.
\end{ex}

\begin{proof} (of Theorem 1.4)
We prove the equivalent assertion which is obtained by replacing $X$ by
$RST(\lambda, k)$ and $X(q)$ by
$q^{-\kappa(\lambda')} s_{\lambda'}(1,q,q^2,\dots,q^{k-1})$.
We fix a standard tableau $T$ of shape $\lambda'$ and consider the action of
$(1,2,\dots,k)$ on $V_{T,k}$.

Let $\zeta = e^{\frac{2 \pi i}{k}} \in \mathbb{C}$.  Suppose $U \in
RST(\lambda, k, \alpha), \alpha = (\alpha_1, \dots, \alpha_k)$ satisfies
$c_k^m \cdot I_{\alpha}(U) = I_{\alpha}(U)$.  Since $c_k$ maps elements of the
form $I_{(\alpha_1, \dots, \alpha_k)}(P)$
for $P \in RST(\lambda, k, (\alpha_1, \dots, \alpha_k))$ to elements of the
form $I_{(\alpha_k, \alpha_1, \dots, \alpha_{k-1})}(Q)$ for
$Q \in RST(\lambda, k, (\alpha_k, \alpha_1, \dots, \alpha_{k-1}))$, it follows
that $i \equiv j$ (mod $m$) implies that $\alpha_i = \alpha_j$.  Using the
facts that $\kappa(\lambda') = \frac{ab(b-1)}{2}$ and $\alpha_1 + \dots +
\alpha_k = n = ab$ we get that
\begin{align*}
(\zeta^m)^{\kappa(\lambda')} &=
(e^{\frac{2 \pi i}{k}})^{\frac{mab(b-1)}{2}} \\
&= (e^{\frac{\pi i}{k}})^{mn(b-1)} \\
&= (e^{\pi i})^{(\alpha_1 + \dots + \alpha_m)(b-1)}\\
&= (-1)^{(\alpha_1 + \dots + \alpha_m)(b-1)}.\\
\end{align*}
However, we already know that
\begin{equation*}
c_k^m I_{\alpha}(U) = (-1)^{(\alpha_1 + \dots + \alpha_m)(b-1)}
I_{\alpha}(j^m(U)).
\end{equation*}

It follows that for any $r \geq 0$, the coefficient of $I_{\alpha}(U)$ in
$c_k^r I_{\alpha}(U)$ is
$\zeta^{r \kappa(\lambda')}$ if $j^r(U) = U$ and $0$ otherwise.  Therefore, the
trace of the operator $c_k^r$ on the space $V_{T,k}$ is equal to
$\zeta^{r \kappa(\lambda')} |X^{j^r}|$.

On the other hand, the operator $c_k^r$ is conjugate to
diag($1, \zeta^r, \zeta^{(2r)}, \dots, \zeta^{(k-1)r}$) in $GL_k(\mathbb{C})$.
It is easy to see that if $P$ is row strict with content $\alpha$, then
$I_{\mu}(P)$ is an eigenvector for the latter operator with eigenvalue
$(1^r)^{-\alpha_1} (\zeta^r)^{-\alpha_2} (\zeta^{2r})^{-\alpha_3} \cdots$.  It
follows that the trace of
$c_k^r$ on $V_{\lambda}$ is the specialization of the Schur function
$s_{\lambda'}(1, q, q^2, \dots, q^{k-1})$ at $q = \zeta^{-r}$.  The desired CSP
follows.
\end{proof}

As a pair of corollaries to Theorem 1.4, we can prove 
Reiner-Stanton-White's results from the introduction quite
easily.

\begin{proof} (of Theorem 1.1)
Interchange $n$ and $k$ in Theorem 1.4 and take $\lambda$ to be a single
column.
\end{proof}

\begin{proof} (of Theorem 1.2)
Interchange $n$ and $k$ in Theorem 1.4 and take $\lambda$ to be a single row.
\end{proof}

\section{Promotion on Column Strict Tableaux with Fixed Content}
In this section we fix a rectangular partition $\lambda = b^a$ with $ab = n$, a
positive integer $k$ and a composition $\alpha \models n$ with $\ell(\alpha) =
k$ such that $\alpha$ has some cyclic symmetry.  We prove a near-CSP involving
the action of certain powers of promotion on the set $CST(\lambda, k, \alpha)$
and the Kostka-Foulkes polynomials.  In representation theoretic terms, this
corresponds to a weight space refinement of our results in the last section.

Assume that for some integer $d | k$, we have that the composition $\alpha$ has
cyclic symmetry of order $d$.  That is, $\alpha_i = \alpha_j$ whenever $i
\equiv j$ (mod $d$).  Since $j$ maps the set $CST(\lambda, k, (\alpha_1, \dots,
\alpha_k))$ into the set $CST(\lambda, k, (\alpha_k, \alpha_1, \alpha_2, \dots,
\alpha_{k-1}))$, we have that the $d^{th}$ power $j^d$ of $j$ maps the set
$CST(\lambda, k, \alpha)$ into itself.  Note that for the special case $d = 1$,
$k = n$, and $\alpha = 1^n$ this is the statement that $j$ acts on the set
$SYT(\lambda)$ of standard tableaux of shape $\lambda$.  Since $j$ acts with
order $k$ on the set $CST(\lambda, k)$, we have that $j^d$ generates an action
of the cyclic group $\mathbb{Z} / (\frac{k}{d} \mathbb{Z})$ on $CST(\lambda, k,
\alpha)$.

For a partition $\lambda \vdash n$ and a composition $\alpha \models n$, let
$K_{\lambda, \alpha}(q) \in \mathbb{N}[q]$ be the associated Kostka-Foulkes
polynomial.  The Kostka-Foulkes polynomials are $q$-analogues of the Kostka
numbers $K_{\lambda, \alpha}$ which enumerate the number of column strict
tableaux with shape $\lambda$ and content $\alpha$.  In particular $K_{\lambda,
\alpha}(1) = K_{\lambda, \alpha}$ always.
Moreover, we have the polynomial equality $K_{\lambda, \alpha}(q) = K_{\lambda,
\alpha'}(q)$ for any rearrangement $\alpha'$ of the composition $\alpha$.
The Kostka-Foulkes polynomials are the generating function for the charge
statistic on tableaux and are also the coefficients of the change of basis
matrix from Schur functions to Hall-Littlewood symmetric functions.
For more details on these polynomials, see \cite{LLTKF}.  Up to a power of $q$,
the Kostka-Foulkes polynomials will play the role of $X(q)$ in our CSP.

It should be noted that the Kostka-Foulkes polynomials have an interesting
representation theoretic interpretation.  Let $\mathfrak{g}$ denote the simple
Lie algebra $\mathfrak{sl}_k(\mathbb{C})$ and let $\mathfrak{g} \cong
\mathfrak{n}^{-} \oplus \mathfrak{h} \oplus \mathfrak{n}^{+}$ denote the Cartan
decomposition of $\mathfrak{g}$.  For any partition $\lambda \vdash n$, let
$V^{\lambda}$ denote the irreducible representation of $\mathfrak{g}$ indexed
by $\lambda$.  The module $V^{\lambda}$ has a weight space decomposition
$V^{\lambda} \cong \oplus_{\mu} V^{\lambda}_{\mu}$, where the $\mu$ are
elements of the dual algebra $\mathfrak{h}^{*}$.  Let $X$ be a \emph{generic}
element of the nilpotent subalgebra $\mathfrak{n}^{+}$ of $\mathfrak{g}$
generated by the positive roots.  For $j = 0, 1, 2, \dots$, let $V_j$ denote
the subspace of the weight space $V^{\lambda}_{\mu}$ which is killed by
$X^{j}$, so that $V_0 \subseteq  V_1 \subseteq V_2 \subseteq \dots$.  Since $X$
is an element of the positive subalgebra, some sufficiently high power of $X$
must carry every element of $V^{\lambda}_{\mu}$ outside of the weight support
for the representation $V^{\lambda}$, so this filtration terminates in
$V^{\lambda}_{\mu}$ for large $j$.  This filtration of the weight space
$V^{\lambda}_{\mu}$ is called the \emph{Brylinski-Kostant (BK) filtration}.  It
turns out that the Kostka-Foulkes polynomial $K_{\lambda,\mu}(q)$ is equal to
the associated jump polynomial (\cite{Brylinski}, \cite{JLZ}), that is,
\begin{equation}
K_{\lambda,\mu}(q) = \sum_{j \geq 0} \dim(V_{j+1} / V_j) q^j.
\end{equation}

To get our fixed point result, we will need some information about the
evaluation of Kostka-Foulkes polynomials at roots of unity.  Lascoux, Leclerc,
and Thibon \cite{LLTKF} have interpreted these evaluations in terms of ribbon
tableaux.  We define the relevant combinatorial objects.

For a positive integer $m$, an \emph{m-ribbon} is a connected skew shape with
$m$ boxes which contains no $2$ by $2$ squares.  The southwesternmost box of an
$m$-ribbon is called the \emph{head} of the $m$-ribbon and the northeasternmost
box of an $m$-ribbon is called the \emph{tail} of the $m$-ribbon.

For any skew partition $\mu / \nu$, an \emph{m-ribbon tableau} of shape $\mu /
\nu$ is a tiling of the diagram of $\mu / \nu$ by $m$-ribbons with a number
attached to every $m$-ribbon in the tiling.  Observe that the set of $m$-ribbon
tableaux of shape $\mu$ is empty unless $m$ divides $|\mu / \nu|$.  The content
of an $m$-ribbon tableau $T$ of shape $|\mu / \nu|$ is the composition of $|\mu
/ \nu| / m$ given by $($number of 1s in $T$, number of 2s in $T$, $\dots )$.
If an $m$-ribbon tableau $T$ of shape $\mu / \nu$ exists, the sign
$\epsilon_m(\mu / \nu)$ of $\mu / \nu$ is the number $(-1)^h(T)$, where $h(T)$
is the sum of the heights of the ribbons in $T$.  It can be shown that the sign
of a skew shape $\mu / \nu$ is independent of the tableau $T$ chosen.

An $m$-ribbon tableau $T$ of shape $\mu / \nu$ is said to be \emph{column
strict} if for any ribbon $R$ in $T$ with label, $j$ the head of $R$ does not
lie to the right of a ribbon with label $i > j$ and the tail of $R$ does not
lie below a ribbon with label $i \geq j$.  Observe that for $m = 1$ this
definition reduces to the ordinary definition of column strict tableaux.  For
any composition
$\beta \models \frac{|\mu / \nu|}{m}$, let $K^{m}_{\mu/\nu, \beta}$ denote the
number of column strict ribbon tableaux of shape $\mu/\nu$ and content $\beta$.

Say that a skew shape $\mu / \nu$ is a \emph{horizontal m-ribbon strip} if
there exists a column strict $m$-ribbon tableau of shape $\mu / \nu$ in which
every ribbon has the same label.  If $\mu / \nu$ is a horizontal $m$-ribbon
strip, it is easy to see that there exists a unique such column strict ribbon
tableau (for fixed choice of label).  The following result is mentioned on page
12 of \cite{LLTKF}.

\begin{thm}(Lascoux-Leclerc-Thibon \cite{LLTKF})
Let $\lambda \vdash n$ be a partition and $\alpha \models n$ be a composition.
For $d | n$, let $\zeta \in \mathbb{C}$ be a root of unity of order $d$.  If
the multiplicity of any part of $\alpha$ is not divisible by $d$, then
$K_{\lambda, \alpha}(\zeta) = 0$.  If the multiplicity of every part of
$\alpha$ is divisible by $d$, then the modulus $|K_{\lambda, \alpha}(\zeta)|$
is equal to the number of column strict $d$-ribbon tableaux with content
$\tilde{\alpha}$.  Here $\tilde{\alpha}$ is any composition of $\frac{n}{d}$
whose part multiplicities are all $\frac{1}{d}$ times the part multiplicities
of $\alpha$.
\end{thm}

We will also need a minor lemma which relates Schur function specialization to
ribbon tableau enumeration.  For this lemma there is no need for our
rectangular shape hypothesis.

\begin{lem}
Let $\lambda \vdash n$ be a partition of \emph{arbitrary} shape.  Let $k > 0$
be a positive integer and let $d | k$.  Let $\zeta = e^{2 \pi i / k}$ be a
primitive $k^{th}$ root of unity and let $a_1, \dots, a_d$ be arbitrary complex
numbers.

We have that
\begin{equation}
(\zeta^d)^{\kappa(\lambda)} s_{\lambda}(a_1, \zeta^d a_1, \dots, \zeta^{k-d}
a_1, \dots, a_d, \dots, \zeta^{k-d} a_d) =
\sum_{\tilde{\beta}} K^{\frac{k}{d}}_{\lambda,\tilde{\beta}}
a_1^{\frac{k}{d}\tilde{\beta}_1} \cdots a_d^{\frac{k}{d}\tilde{\beta}_d},
\end{equation}
\noindent
where the sum on the right hand side is over all compositions $\tilde{\beta}
\models \frac{nd}{k}$.
\end{lem}

The `domino' case $d = \frac{k}{2}$ of the above lemma was known, for example,
to Stembridge \cite{StemTab}.  We also remark that the right hand side of the
above identity can be interpreted as a product of specializations of Schur
functions.  In particular let $(\lambda^{(1)}, \lambda^{(2)}, \dots,
\lambda^{(\frac{k}{d})})$ be the $\frac{k}{d}$-quotient of the partition
$\lambda$ and let $\nu$ be the $\frac{k}{d}$-core of $\lambda$.  By the
Stanton-White correspondence \cite{SW} there exists a content preserving
bijection from column strict $\frac{k}{d}$-ribbon tableaux of skew shape
$\lambda / \nu$ and $\frac{k}{d}$-tuples of ordinary column strict tableaux of
shapes $(\lambda^{(1)}, \lambda^{(2)}, \dots, \lambda^{(\frac{k}{d})})$.  From
this bijection it is fairly easy to derive an alternative formulation of Lemma
6.2.  Namely, the expression $s_{\lambda}(a_1, \zeta^d a_1, \dots, \zeta^{k-d}
a_1, \dots, a_d, \dots, \zeta^{k-d} a_d)$ is equal to 0 unless $\lambda$ has
empty $\frac{k}{d}$-core and, if $\lambda$ does have empty $\frac{k}{d}$-core,
we have the identity
\begin{equation}
| s_{\lambda}(a_1, \zeta^d a_1, \dots, \zeta^{k-d} a_1, \dots, a_d, \dots,
| \zeta^{k-d} a_d)| = | \Pi_{i=1}^{\frac{k}{d}} s_{\lambda^{(i)}}(a_1, a_2,
| \dots, a_d)|.
\end{equation}

\begin{proof}
We show that both sides of this equation satisfy the same recursion.

For the left hand side of the above equation, observe that in any column strict
tableaux in $CST(\lambda, k)$, the last $\frac{k}{d}$ letters $\{ k - d + 1, k
- d + 2, \cdots k \}$ of $[k]$ must appear in a skew shape flush with the
exterior of the shape $\lambda$.  This gives a partition of $CST(\lambda, k)$
according to this skew shape $\lambda / \nu$.

We consider the skew Schur function specialization
$s_{\lambda / \nu} (1, \zeta^d , \dots, \zeta^{k-d})$ for the various skew
shapes $\lambda / \nu$.  By a result of Lascoux, Leclerc, and Thibon
\cite{LLTKF}, we have that the above is equal to 0 if $\lambda / \nu$ is not a
horizontal $\frac{k}{d}$-ribbon strip and is equal to
$\epsilon_{\frac{k}{d}}(\lambda / \nu)$ if $\lambda / \nu$ is a horizontal
$\frac{k}{d}$-ribbon strip.

For the right hand side, given any column strict $\frac{k}{d}$-ribbon tableau
$T$ of content $\tilde{\beta} = (\tilde{\beta_1}, \dots, \tilde{\beta_d})$, we
have that the restriction of $T$ to the number $\tilde{\beta_d}$ is a
horizontal $\frac{k}{d}$-ribbon strip of skew shape $\lambda / \nu$ for some
normal shape $\nu \subseteq \lambda$.  On the other hand, given any normal
shape $\nu \subseteq \lambda$ such that $\lambda / \nu$ is a horizontal
$\frac{k}{d}$-ribbon strip of with $\tilde{\beta_k}$ ribbons and any column
strict $\frac{k}{d}$-ribbon tableau $T_o$ of shape $\nu$ and content
$(\tilde{\beta_1}, \tilde{\beta_2}, \dots, \widetilde{\beta_{k-1}})$, we have
that $T_o$ extends uniquely to a column strict $\frac{k}{d}$-ribbon tableau $T$
of shape $\lambda$ and content $\tilde{\beta}$.  It follows that both sides of
the above equation satisfy the same recursion relation and, thus, the above
equation holds for arbitrary partitions $\lambda$.  Also, we have that both
sides of the above equation are equal to zero when $\lambda$ is not a
$\frac{k}{d}$-ribbon tableau.
\end{proof}

\begin{proof} (of Theorem 1.5)

For any composition $\beta \models n$ of length $k$, let $N^{d}_{\lambda,
\beta}$ denote the number of column strict tableaux of shape $\lambda$ and
content $\beta$ which are fixed by the operator $j^d$.  Observe that
$N^{d}_{\lambda, \beta}$ is nonzero only if $\beta_i = \beta_j$ whenever $i
\equiv j$ (mod $d$).

Let $c_k = (1, 2, \dots, k)$ denote the long cycle in $S_k \subset
GL_k(\mathbb{C})$.
Let $T$ be an arbitrary standard tableau of shape $\lambda'$ and let $V_{T,k}$
be the irreducible $GL_k(\mathbb{C}$-representation constructed in Section 4.
Identify the basis elements of $V_{T,K}$ with symbols $I_{\beta}(P)$ for
compositions $\beta$ of $n$ of length $k$ and \emph{column} strict tableaux $P$
of shape $\lambda$ and content $\beta$ in the obvious way.  By Theorem 4.6 have
a weight space decomposition
\begin{equation*}
V_{T,k} \cong \bigoplus_{\beta} V_{T,k,\beta},
\end{equation*}
where $V_{T,k,\beta}$ is the $\mathbb{C}$-linear subspace of $V_{T,k}$ spanned
by the vectors $I_{\beta}(P)$ for $P \in CST(\lambda, k, \beta)$.  Recall that
each of the subspaces $V_{T,k,\beta}$ is stabilized by the Cartan subgroup $H$
of diagonal matrices in $GL_k(\mathbb{C})$ and the action of $c_k$ maps the
space $V_{T,k,\beta}$ into the space $V_{T,k,c_k \cdot \beta}$.

Let $\zeta = e^{\frac{2 \pi i}{k}}$ and let $a_1, \dots, a_d \in
\mathbb{C}^{\times}$ be any nonzero complex numbers.  From the discussion in
the above paragraph and Theorem 4.6, the action of the composition of $c_k^d$
with an element of the Cartan subgroup yields the following character
evaluation.

\begin{align}
\epsilon_{\frac{k}{d}}(\lambda){\kappa(\lambda)} s_{\lambda}(a_1, \zeta^d a_1,
\dots, \zeta^{k-d} a_1, \dots, a_d, \dots, \zeta^{k-d} a_d) &= \\
\epsilon_{\frac{k}{d}}(\lambda){\kappa(\lambda)}
[s_{\lambda}(x_1,\dots,x_k)]_{x_i = \zeta^{d i} a_{\lfloor i/d \rfloor}} &=\\
\sum_{\beta} N^d_{\lambda,\beta} a_1^{\frac{k}{d}\beta_1} \cdots
a_d^{\frac{k}{d}\beta_d},
\end{align}

\noindent
where the sum on the lattermost expression ranges over all compositions $\beta$
of $n$ of length $k$ such that $\beta_i = \beta_j$ whenever $i \equiv j$ (mod
$d$).
By Lemma 6.2 we can interpret the first expression in the above sequence of
equalities as follows.

\begin{equation}
\sum_{\beta} N^d_{\lambda,\beta} a_1^{\frac{k}{d}\beta_1} \cdots
a_d^{\frac{k}{d}\beta_d} = \sum_{\tilde{\beta}}
K^{\frac{k}{d}}_{\lambda,\tilde{\beta}} a_1^{\frac{k}{d}\tilde{\beta}_1} \cdots
a_d^{\frac{k}{d}\tilde{\beta}_d},
\end{equation}

where the sum on the left hand side ranges over all compositions $\beta$ of $n$
having length $k$ such that $\beta_i = \beta_j$ whenever $i \equiv j$ (mod $d$)
and the sum on the right hand side ranges over all compositions $\tilde{\beta}$
of $\frac{nd}{k}$ having length $d$.  Choosing the complex numbers $a_1, \dots,
a_d$ to be algebraically independent over $\mathbb{Q}$, we can view this latter
equation as a polynomial identity.  It follows that if $\beta$ is any
composition of $n$ of length $k$ satisfying $\beta_i = \beta_j$ whenever $i
\equiv j$ (mod $d$) and if we set
$\tilde{\beta} := (\beta_1, \beta_2, \dots, \beta_d)$, we have the equality

\begin{equation*}
N^d_{\lambda,\beta} = K^{\frac{k}{d}}_{\lambda,\tilde{\beta}}.
\end{equation*}

It remains to relate the right hand side of the above equation to the
evaluation of a Kostka-Foulkes polynomial at an appropriate root of unity. This
is Theorem 6.1.
\end{proof}

\begin{rem}
It would be desirable to strengthen Theorem 1.5 to a genuine CSP.  One obvious
way of doing this would be to find an expression $f(\lambda, \alpha)$
potentially depending on both the rectangular partition $\lambda$ and the
composition $\alpha$ such that the triple $(CST(\lambda, k, \alpha), \langle
j^d \rangle, q^{f(\lambda,\alpha)} K_{\lambda,\alpha}(q))$ exhibits the CSP.
That is, our polynomial $X(q)$ would just be a $q-$shift of a Kostka-Foulkes
polynomial.  By Lascoux-Leclerc-Thibon we know that the evaluation of the
Kostka-Foulkes polynomial $K_{\lambda,\alpha}(q)$ itself at the relevant roots
of unity is plus or minus the correct positive number, but we do not know of an
expression $f(\lambda,\alpha)$ which takes this sign into account.
\end{rem}

\section{Dihedral Actions}
The cyclic sieving phenomenon is concerned with the enumeration of the fixed
point set sizes when we
are given an action on a finite set $X$ by a finite cyclic group $C$.  However,
there are many
interesting combinatorial actions on finite sets where the group in question is
not cyclic.  In general, given a finite group $G$ acting on a finite set $X$,
we can ask for the fixed point set sizes
$|X^{g}|$ for all $g \in G$.  For the case where $G$ is a product of two cyclic
groups, Barcelo, Reiner, and Stanton have proven an instance of a
\emph{bicyclic} sieving phenomenon involving complex reflection groups where
the fixed point set sizes involved are obtained by evaluating a polynomial
$X(q,t) \in \mathbb{Q}[q,t]$ in two variables at two appropriate roots of
unity.  In our case, we are interested in actions of \emph{dihedral} groups $D$
generated by promotion and evacuation on either of the sets $SYT(\lambda)$ or
$CST(\lambda, k)$, where $\lambda$ is a fixed rectangular shape and $k \geq 0$.
While it is not yet known how to interpret the associated fixed point set sizes
as evaluations of a naturally associated polynomial, we hope that this work
will give motivation to the study of a `dihedral sieving phenomenon'.  Some
work regarding dihedral actions is in the REU report \cite{AKLM}.

Let $\lambda = b^a$ be a rectangular partition with $ab = n$ and let $k \geq
0$.  It is possible to show that we have $e j e = j^{-1}$ as operators on
$CST(\lambda, k)$.  Moreover, we know that the order of the operator $j$ is
equal to $k$ while the order of the operator $e$ is equal to $2$.  This implies
that the group $\langle e, j \rangle$ generated by $e$ and  $j$, considered as
a subgroup of the symmetric group $S_{CST(\lambda, k)}$, is dihedral of order
$2k$.  Similar remarks hold for the action of $e$ and $j$ on $RST(\lambda, k)$
and $SYT(\lambda)$.

We have already determined that the action of $j$ on the above sets of tableaux
is modeled by the action of the long cycle of an appropriate symmetric group on
an appropriate module.  By Stembridge's result, the action of $e$ can be
modeled by the action of the long element $w_o$.  So, corresponding to the
dihedral group acting on the sets $CST(\lambda, k)$, $RST(\lambda, k)$, and
$SYT(\lambda)$, we get an isomorphic dihedral group generated by the long cycle
and $w_o$ sitting inside $S_k$, $S_k$, and $S_n$, respectively.
Our first observation gives the cycle types of the operators $w_o$ and $w_o
c_n$ in the symmetric group $S_n$.
\begin{lem}
Let $w_o \in S_n$ be the long element and let $c_n \in S_n$ be the long cycle
$(1, 2, \dots, n)$. The permutation  $w_o$ has cycle type $2^{\frac{n}{2}}$
when $n$ is even and $2^{\frac{n-1}{2}}1$ when $n$ is odd.  The permutation
$w_o c_n$ has cycle type $2^{(\frac{n}{2}-1)}1^2$ when $n$ is even and
$2^{\frac{n-1}{2}}1$ when $n$ is odd.
\end{lem}

\begin{figure}
\centering
\includegraphics[scale=.5]{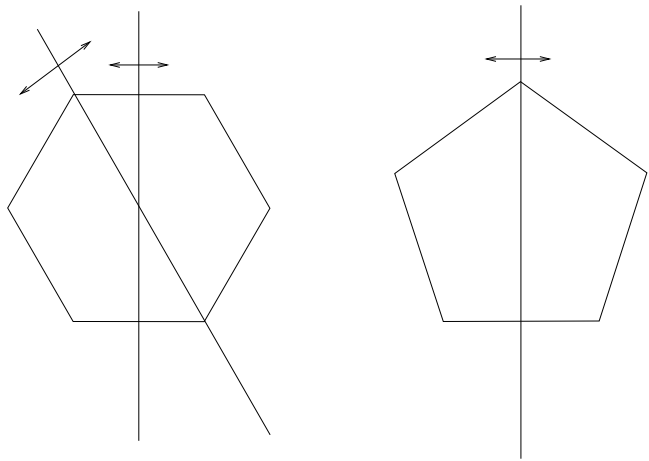}
\caption{Conjugacy classes of reflections in dihedral groups}
\end{figure}

\begin{proof}
This is straightforward from the definition of permutation multiplication.
Looking at the dihedral group generated by $c_n$ and $w_o$, this essentially
asserts that if $n$ is odd, any element of the unique conjugacy class of
reflections within the symmetry group of an $n$-gon fixes one vertex and if $n$
is even there are two conjugacy classes of reflections fixing zero or two
vertices.  See Figure 7.1.
\end{proof}
This simple lemma about permutation multiplication allows us to pin down,
including sign, the action of the operators $w_o$ and $w_o c_n$ on the cellular
representation.  As in Section 3, given a rectangular partition $\lambda =
b^a$, identify the basis of the KL cellular representation of $S_n$
corresponding to $\lambda$ with $SYT(\lambda)$.
\begin{prop}
Let $\lambda = b^a$ be a rectangle and let $\rho$ be the KL cellular
representation corresponding to $\lambda$.  We have that
\begin{align}
\rho(w_o) P &=
\begin{cases}
e(P) & \text{if $b$ is even} \\
(-1)^{\lfloor \frac{a}{2} \rfloor}e(P) & \text{if $b$ is odd}
\end{cases}\\
\rho(w_o c_n) P &=
\begin{cases}
e(j(P)) & \text{if $b$ is even} \\
(-1)^{(\frac{a}{2}-1)}e(j(P)) & \text{if $b$ is odd and $a$ is even} \\
(-1)^{\lfloor \frac{a}{2} \rfloor}e(j(P)) & \text{if $b, a$ are odd.}
\end{cases}
\end{align}
\end{prop}
\begin{proof}
Recall that the KL cellular representation corresponding to a shape $\lambda$
is an irreducible $S_n$ module of shape $\lambda$.
By Theorem 2.10 and Proposition 3.5, we need only verify that the signs in
these equations are correct.  To do this we will use the Murnaghan-Nakayama
rule (see, for example, \cite{Sag}).  For any of the $3$ cycle types which
appear in Lemma 7.1, there exists a corresponding removal of rim hooks from the
corresponding rectangle with the appropriate sign.

In particular, if $b$ is even, we can fill the diagram of $\lambda$ with
$\frac{n}{2}$ rim hooks of shape $2 \times 1$ in the case of $w_o$ and
$(\frac{n}{2} - 1)$ rim hooks of shape $2 \times 1$  together with two rim
hooks of shape $1 \times 1$  in the case of $w_o c$.  In either case, the
product $(-1)^{h(\nu)}$ over all the rim hooks $\nu$ in either of these
fillings, where $h(\nu)$ is the height of $\nu$, is equal to $1$.

If $b$ is odd and $a$ is even for the case of $w_o$, we can fill the diagram of
$\lambda$ with $\frac{n-a}{2}$ rim hooks of shape $2 \times 1$ in the left
$(b-1) \times a$ box and $\frac{a}{2}$ rim hooks of shape $1 \times 2$ in the
first column.  The sign of this filling is equal to $(-1)^{\frac{a}{2}}$.  In
the case of $w_o c$ with the same parity conditions on $a$ and $b$, we can
replace the upper left $1 \times 2$ rectangle in the filling with two $1 \times
1$ rim hooks, and the resulting filling has sign $(-1)^{(\frac{a}{2}-1)}$.

If $b$ and $a$ are both odd, in the case of $w_o$ and $w_o c_n$ we can fill the
diagram of $\lambda$ with $\frac{n-a}{2}$ rim hooks of shape $2 \times 1$ in
the left $(b-1) \times a$ box, $\lfloor \frac{a}{2} \rfloor$ rim hooks of shape
$1 \times 2$ in the lower portion of the first column of $\lambda$, and one $1
\times 1$ rim hook in the upper left hand corner of $\lambda$.  The sign of
this filling is $(-1)^{\lfloor \frac{a}{2} \rfloor}$.

It is possible to show that the sign of all rim hook fillings of $\lambda$
corresponding to the conjugacy class of either $w_o$ or $w_o c_n$ is the same,
either $1$ or $-1$.  The desired result now follows from the Murnaghan-Nakayama
rule.
\end{proof}
Observe that, in this case, we were able to use the Murnaghan-Nakayama rule to
deduce the sign of our permutation matrices relatively easily.  This is in
contrast to the proof of Lemma 3.4, where we needed to use special facts about
the KL cellular representation.  In fact, it is possible to use the
Murnaghan-Nakayama rule in a similar way to derive Lemma 3.4 in the cases where
both $a$ and $b$ are not odd, but for this particular case more specific facts
are needed.

We already know that the number of points fixed by the action of a given power
of $j$ on standard tableaux is given by the evaluation of the $q$-hook length
formula at an appropriate root of unity.  The previous lemma allows us to give
the number of fixed points of the operators $e$ and $ej$ on standard tableaux
as a character evaluation.
\begin{prop}
Let $\lambda = b^a$ be a rectangular partition of $n$, let $c$ be the long
cycle in $S_n$, and let $\chi^{\lambda}: S_n \rightarrow \mathbb{C}$ be the
irreducible character of $S_n$ corresponding to $\lambda$.  Let $\langle e, j
\rangle$ be the dihedral subgroup of $S_{SYT(\lambda)}$ generated by $e$ and
$j$ and let $\phi : \langle w_o, c_n \rangle \rightarrow \langle e, j \rangle$
be the group epimorphism defined by $\phi(w_o) = e$ and $\phi(c) = j$.  For any
$g \in \langle e, j \rangle$ and any $w$ such that $\phi(w) = g$ one has
\begin{equation*}
| SYT(\lambda)^{g}| = \pm \chi^{\lambda}(w).
\end{equation*}

More precisely, we have that
\begin{align}
| SYT(\lambda)^{e}| &=
\begin{cases}
\chi^{\lambda}(w_o) & \text{if $b$ is even}\\
(-1)^{\lfloor \frac{a}{2} \rfloor}\chi^{\lambda}(w_o)
& \text{if $b$ is odd}
\end{cases}\\
| SYT(\lambda)^{ej}| &=
\begin{cases}
\chi^{\lambda}(w_o c_n) & \text{if $b$ is even}\\
(-1)^{(\frac{a}{2}-1)} \chi^{\lambda}(w_o c_n)
& \text{if $b$ is odd and $a$ is even}\\
(-1)^{\lfloor \frac{a}{2} \rfloor} \chi^{\lambda}(w_o c_n)
& \text{if $b$ and $a$ are odd.}
\end{cases}
\end{align}
\end{prop}
\begin{proof}
This follows easily from the above lemma.
\end{proof}
Our next task is to extend these results to facts about the dual canonical
basis.  Let $\lambda$ be a rectangular partition and let $k \geq 0$.  By the
results of Section 4, there exists a basis $\{ I_{\alpha}(P) \}$ for the
$GL_k(\mathbb{C})$-module $V_{T, k}$ for $T \in SYT(\lambda)$ fixed, where
$\alpha$ ranges over all compositions of $n$ of length $k$ and $P$ ranges over
all elements of $RST(\lambda, k, \alpha)$.  For any choice of $\alpha$ and $P$
and any diagonal matrix $g =$ diag($x_1, \dots, x_k$) $\in GL(\mathbb{C}^k)$,
the vector
$I_{\alpha}(P)$ is an eigenvector for $g$ with eigenvalue
$x_1^{-\alpha_1} x_2^{-\alpha_2} \cdots$.  We first show how multiplication by
the long elements interacts with the projection maps $\pi_{\alpha}$.
\begin{lem}
Let $\lambda$ be a rectangular partition and fix $T \in SYT(\lambda)$.
Let $w_{o_n}$ be the long permutation in $S_n$ and let $w_{o_k}$ be the long
permutation in $S_k$.  We have the following commutative diagram, where $\alpha
= (\alpha_1, \dots, \alpha_k) \models n$ and
$w_{o_k} \cdot \alpha$ is the composition of $n$ of length $k$ given by
$(\alpha_k, \alpha_{k-1}, \dots, \alpha_1)$.
\begin{equation}
\begin{array}{ccccc}
& & w_{o_n} & & \\
& V_{T, n, 1^n} & \longrightarrow & V_{T,n,1^n} & \\
\pi_{\alpha} & \downarrow & & \downarrow &
\pi_{w_{o_k} \cdot \alpha}\\
& V_{T,k,\alpha} & \longrightarrow &
V_{T,k, w_{o_k} \cdot \alpha} & \\
& & w_{o_k} & &
\end{array}
\end{equation}
\end{lem}
\begin{proof}
The commutativity of the corresponding diagram of polynomial rings is easily
verified by computing the image under the two possible compositions of the
generators $x_{ij}$, as before.  The commutativity of the above diagram is
induced from this other commutative diagram.
\end{proof}
Using this lemma, we obtain the action of the long element in $S_k$ and its
product with the long cycle $S_k$ on the basis elements $I_{\alpha}(P)$, $P \in
RST(\lambda, k, \alpha)$.
\begin{lem}
Let $w_{o_k}$ be the long element in $S_k$, let $c_k$ be the long cycle $(1,2,
\dots, k)$ in $S_k$, and let $\lambda = b^a$ be a rectangular partition, and
let $P \in RST(\lambda, k, \alpha)$.  We have that
\begin{align*}
w_{o_k} I_{\alpha}(P) &=
\begin{cases}
I_{(\alpha_k, \dots, \alpha_1)}(e(P)) & \text{if $a$ is even} \\
(-1)^{\lfloor \frac{b}{2} \rfloor}
I_{(\alpha_k, \dots, \alpha_1)}(e(P)) & \text{if $a$ is odd}
\end{cases}\\
w_{o_k} c_k I_{\alpha}(P) &=
\begin{cases}
(-1)^{\alpha_k(b-1)}I_{(\alpha_{k-1}, \alpha_{k-2}, \dots, \alpha_1,
\alpha_k)}(ej(P)) & \text{if $a$ is even} \\
(-1)^{(\frac{b}{2}-1 + \alpha_k(b-1))}
I_{(\alpha_{k-1}, \alpha_{k-2}, \dots, \alpha_1, \alpha_k)}(ej(P)) &
\text{if $a$ is odd and $b$ is even} \\
(-1)^{(\lfloor \frac{b}{2} \rfloor + \alpha_k(b-1))}
I_{(\alpha_{k-1}, \alpha_{k-2}, \dots, \alpha_1, \alpha_k)}
(ej(P)) & \text{if $b, a$ are odd.}
\end{cases}
\end{align*}
\end{lem}
\begin{proof}
This combines previously proven results about the actions of the long element
and the long cycle.
\end{proof}
Our results about fixed points have two flavors.  When $k$ is odd, $j$ and $ej$
are conjugate in the dihedral group
$\langle e, j \rangle$ and hence have the same number of fixed points as
operators on $CST(\lambda,k)$.  However, when $k$ is odd, $e$ and $ej$ are not
conjugate, so different numbers of fixed points are possible.  Parts of the
next theorem come from conjectures and results of
Abuzzahab, Korson, Li, and Meyer \cite{AKLM} and a theorem of Stembridge
\cite{StemTab}.
\begin{thm}
Let $\lambda = b^a$ be a rectangular partition of $n$ and let $k \geq 0$.
Assume that at $k$ is odd.  We have that
\begin{align}
| CST(\lambda,k)^{e}| &= (-1)^{\kappa(\lambda)}s_{\lambda}(1,-1,1, \dots,
| (-1)^{k-1})\\
&= |CST(\lambda,k)^{ej}|.
\end{align}
On the other hand, if $k$ is even, we have that
\begin{align}
| CST(\lambda,k)^{e}| &= (-1)^{\kappa(\lambda)}s_{\lambda}(1,-1,1, \dots,
| (-1)^{k-1})\\
| CST(\lambda,k)^{ej}| &=
\begin{cases}
(-1)^{\kappa(\lambda)}s_{\lambda}(1,-1,\dots,(-1)^{k-3},(-1)^{k-2},(-1)^{k-2})
 & \text{$a$ and $b$ even}\\
(-1)^{\kappa(\lambda)}s_{\lambda}(1,-1,\dots,(-1)^{k-1}) & \text{$a$ even, $b$
odd}\\
(-1)^{\frac{b}{2}-1+\kappa(\lambda)}
s_{\lambda}(1,-1,\dots,(-1)^{k-1})
& \text{$a$ odd, $b$ even}\\
(-1)^{\lfloor \frac{b}{2} \rfloor + \kappa(\lambda)}
s_{\lambda}(1,-1,\dots,(-1)^{k-3},(-1)^{k-2},(-1)^{k-2})
& \text{$a$ and $b$ odd.}
\end{cases}
\end{align}
\end{thm}
\begin{proof}
The claims about the operator $e$ are Theorem 3.1 of \cite{StemTab}.  The claim
about the operator $ej$ in the case where $k$ is odd follows from the claim
about the operator $j$ by conjugacy, as mentioned above.  We are reduced to the
claim about the operator $ej$ in the case where $k$ is even.  As before, we
prove the equivalent statement for row strict tableaux.  This can be obtained
by copying the above claim \emph{verbatim}, replacing every $RST$ with a $CST$,
every $\lambda$ with a $\lambda'$, and swapping every $b$ for every $a$.

Since $k$ is even, we know that the operator $w_o c_k$ as an element of
$GL(\mathbb{C}^k)$ has eigenvalues $1$ with multiplicity $\frac{k}{2}+1$ and
$-1$ with multiplicity
$\frac{k}{2}-1$.  The claim now follows by examining the signs in Lemma 5.4
with the appropriate parity conditions on $a$ and $b$ in mind.  In particular,
multiplication by terms of the form $(-1)^{\alpha_k}$ corresponds to swapping a
sign of one of the arguments in the appropriate Schur function.
\end{proof}
Observe that in the above `column strict' result one does \emph{not} have that
for any $g$ in the group of operators generated by evacuation and promotion the
fixed point set
$|CST(\lambda, k)^g|$
is equal to plus or minus the character of a general linear group element
corresponding to $g$.  Roughly speaking, the reason for this is that Equation
(5.4) implies that the long cycle acts as promotion on column strict tableaux
up to a sign which is \emph{not} in general independent of weight space.

\section{Applications to other Combinatorial Actions}
For a positive integer $n$, let $B_n$ denote the group of signed permutations
of $[n]$ having order $2^n n!$.  The group $B_n$ is a Coxeter group with
generators $s'_0, s'_1, s'_2, \dots, s'_{n-1}$, where $s'_i$ for $i > 0$ is the
adjacent transposition switching positions $i$ and $i+1$ but preserving signs
and $s'_0$ switches the sign of the element in the first position.  Under this
identification, the long element of $B_n$ is the signed permutation which sends
$1$ to $-1$, $2$ to $-2$, $\dots$, and $n$ to $-n$.  Since the long element
$w_o^{B_n}$ in Type $B_n$ is the scalar transformation $-1$, we have in
particular that $w_o^{B_n} s'_i = s'_i w_o^{B_n}$ for all $i$, so that reduced
expressions for the long element are mapped to other reduced expressions for
the long element under cyclic rotation.   Our cyclic sieving results specialize
in the case of standard tableaux of the square shape $n^n$ to analyze this
cyclic action.
\begin{thm}
Let $X$ be the set of reduced words for the long element in $B_n$.  Let $C =
\mathbb{Z} / n^2 \mathbb{Z}$ act on $X$ by cyclically rotating words.  Let
$X(q) = f^{n^n}(q)$ be the $q-$hook length formula.  We have that the triple
$(X, C, X(q))$ exhibits the cyclic sieving phenomenon.
\end{thm}
\begin{proof}
By Theorem 1.3, we have a CSP involving the cyclic action of promotion on the
set $SYT(n^n)$ and the same polynomial $X(q)$.  Let $Y$ be the set of standard
\emph{shifted} tableaux with $n^2$ boxes and staircase shape (see, for example,
\cite{HaimanInsert}, \cite{Sag2}, or \cite{Worley}).  Shifted jeu-de-taquin
promotion gives a cyclic action on $Y$.

By a result of Haiman (\cite{HaimanInsert}, Proposition 8.11) we have a
bijection between the sets $SYT(n^n)$ and $Y$ under which ordinary promotion
gets sent to shifted promotion.  Moreover, the type $B$ Edelman-Greene
correspondence gives a bijection between $Y$ and the set $X$ under which
shifted jeu-de-taquin promotion maps to cyclic rotation of reduced words
(\cite{Haiman}, Theorem 5.12).
\end{proof}

We can also apply our enumeration results for standard tableaux to prove cyclic
sieving phenomena for handshake patterns and noncrossing partitions.  This
gives a new proof of a result of White \cite{WComm}, as well as a new proof of
results of Heitsch \cite{Heitsch} which have biological applications related to
RNA secondary structure.  The case of handshake patterns gives a very explicit
realization of the dihedral actions investigated in Section 7.

Given $n \in \mathbb{N}$, a \emph{handshake pattern of size 2n} consists of a
circle around which the points $1, 2, \dots, 2n$ are written clockwise and a
perfect matching on the set $[2n]$ such that, when drawn on the circle, none of
the arcs in this matching intersect.  This can be thought of as a way in which
the people labelled $1, 2, \dots, 2n$ can all shake hands so that no one
crosses arms.  Let $H_n$ denote the set of all handshake patterns of size $2n$.
Handshake patterns can be identified with the basis elements for the
Temperley-Lieb algebra $T_n(\zeta)$ (see, for example,
\cite{RSkanTLImm}).

For $n \in \mathbb{N}$, a \emph{noncrossing partition of [n]} is a set
partition $(P_1 | P_2 | \dots | P_k)$ of the set $[n]$ so that whenever there
are integers $a, b, c, d, i,$ and $j$ with $1 \leq a < b < c < d \leq n$ and
$a, c \in P_i$ and $b, d \in P_j$, we must necessarily have that $i = j$.
Noncrossing partitions have a pictorial interpretation.  Drawing the numbers
$1, 2, \dots, n$ clockwise around a circle, a partition $\i$ of $[n]$ is
noncrossing if and only if, when the blocks of $\pi$ are drawn on the circle,
none of the regions intersect.

Recall that a poset $L$ is a lattice if and only if given any pair of elements
$x, y \in L$, there exist unique least upper bounds and greatest lower bounds
for $x$ and $y$ (denoted $x \vee y$ and $x \wedge y$, respectively).  If $L$ is
a finite lattice and $x \in L$, then an element $y \in L$ is said to be a
\emph{complement to x} if we have $x \vee y = \hat{1}$ and also $x \wedge y =
\hat{0}$.  Here $\hat{1}$ and $\hat{0}$ denote the unique maximal and minimal
elements of $L$.  A finite lattice $L$ is said to be \emph{complemented} if
every element has at least one complement.
The set $NC(n)$ of all noncrossing partitions of $[n]$ is a lattice with
respect to the partial order given by refinement (however, this lattice is
\emph{not} a sublattice of the lattice of all partitions since the formulas for
the least upper bounds and greatest lower bounds do not agree).  The lattice
$NC(n)$ is complemented, and given a noncrossing partition $\pi$ of $[n]$,
Kreweras complementation \cite{Kreweras} gives a way to produce a complement of
$\pi$.  Write the numbers $1, 1', 2, 2', \dots, n, n'$ clockwise around a
circle.  Draw the blocks of $\pi$ on the numbers $1, 2, \dots, n$.  Now draw
the unique maximal noncrossing partition of the numbers $1', 2', \dots, n'$
which does not intersect the blocks of $\pi$.  Call this new partition $\pi'$.
Kreweras showed that $\pi'$ is a complement to $\pi$ in $NC(n)$.  Complements
in $NC(n)$ are not in general unique and it is an open problem to determine how
many complements there are in $NC(n)$ for a fixed noncrossing partition $\pi$
of $[n]$.

Both of the sets $H_n$ and $NC(n)$ have cardinality given by the Catalan number
$C_n = \frac{1}{n+1} {2n \choose n}$.  Moreover, these sets both carry an
action of the cyclic group - in the case of $H_n$ given by the action of
rotating the table clockwise by one position and in the case of $NC(n)$ given
by Kreweras complementation.  As a corollary of our earlier work, we get
results of White \cite{WComm} and Heitsch \cite{Heitsch}.
\begin{thm}
Let $n \in \mathbb{N}$ and let $C = \mathbb{Z} / (2n) \mathbb{Z}$.  Let
$X(q)$ be the $q$-Catalan number
\begin{equation*}
X(q) = C_n(q) = \frac{1}{[n+1]_q} {2n \brack n}_q.
\end{equation*}
Let $X$ be either $H_n$ equipped with the $C$-action of rotation or $NC(n)$
equipped with the $C$-action of Kreweras complementation.

Then, the triple $(X, C, X(q))$ exhibits the cyclic sieving phenomenon.
\end{thm}
\begin{proof}
It is well known (see, e.g., Heitsch \cite{Heitsch}) that there exists a
bijection $NC(n) \rightarrow H_n$ under which Kreweras complementation
corresponds to rotation, so we need only prove this theorem in the case where
$X = H_n$.  To do this, we use a bijection $H_n \rightarrow SYT((n,n))$ which
maps rotation to jeu-de-taquin promotion.  It is well known that the
$q$-Catalan numbers are equal to the $q$-hook length formula in the special
case of a $2$ by $n$ partition, so then the result will follow from Theorem
1.3.

To construct this bijection, given a handshake pattern $h \in H_n$, fill the
top row of a $2$ by $n$ tableau $T$ with the smaller members of each of the $n$
handshake pairs in $h$ in increasing order left to right.  Fill the bottom row
of $T$ with the smaller members of each handshake pair in increasing order from
left to right.  Since $h$ is a handshake pattern, it follows that $T$ is a
standard tableau.  It is easy to see that the correspondence $h \mapsto T$
defines a bijection $H_n \rightarrow SYT((n,n))$.

White proved that, under the above bijection, rotation maps to promotion
\cite{WComm}.
For completeness, we prove this again here.

Given a tableau $T \in SYT(n,n)$, define the \emph{ascent set} $A(T)$ of $T$ to
be the subset of $[2n-1]$ defined by $i \in A(T)$ if and only if $i+1$ occurs
strictly north and weakly east of $i$ in $T$.   It is easy to check that for $i
= 1, 2, \dots, 2n-2$, $i$ is contained in $A(T)$ if and only if $i$ is
contained in $A(j(T))$.  Moreover, by Lemma 3.3, $j$ acts cyclically on the
extended descent set.  It is easy to check that the ascent set and extended
descent set of a two-row rectangular standard tableau $T$ together determine
the associated handshake pattern.  Moreover, it is easy to check that table
rotation acts cyclically on the ascent set and extended descent set of the
associated tableau.  Therefore, the bijection between $H_n$ and $SYT((n,n))$
maps table rotation to promotion, as desired.
\end{proof}
White proved the portion of Theorem 8.2 involving handshake patterns directly
in the investigation of the rectangular standard tableaux CSP conjecture, thus
proving the conjecture in the case of tableaux with $2$ rows.  Having proven
this more general conjecture in Theorem 8.2 we get White's result on handshake
patterns as a corollary.

The set $H_n$ also carries an action of the dihedral group $D_{4n}$ of order
$4n$.  Namely, we think of the points $1, 2, \dots 2n$ as the vertices of a
regular $2n$-gon and let $D_{4n}$ act naturally. Let $r$ denote the reflection
in $D_{4n}$ across the line bisecting the pairs of vertices $(1, 2n)$ and $(n,
n+1)$.  Let $s$ denote rotation in $D_{4n}$ by one clockwise unit.
\begin{prop}
Let $X$ be either the set $H_n$ or the set $NC(n)$.  $X$ carries an action of
$D_{4n}$, where $r$ and $s$ have the action described above in the case of
$H_n$ and act by reflection about the line through $1$ bisecting the circle and
Kreweras complementation in the case of $NC(n)$.  Let $w_o$ denote the long
element of $S_{2n}$ and let $c$ denote the long cycle $(1, 2, \dots, 2n)$ in
$S_{2n}$.

Then, the number of fixed points of the operators $r$ and $rs$ on $X$ are given
by the formulas:
\begin{align}
| X^{r}| &=
\begin{cases}
\chi^{(n,n)}(w_o) & \text{if $n$ is even}\\
(-1) \chi^{(n,n)}(w_o)
& \text{if $n$ is odd}
\end{cases}\\
| X^{rs}| &=
\chi^{(n,n)}(w_o c)
\end{align}
\end{prop}
\begin{proof}
It is easy to see, given the simplified algorithm for evacuation in the case of
rectangular tableaux, that under White's bijection $H_n \rightarrow SYT((n,n))$
the action of $r$ corresponds to evacuation.  Moreover, it is simple to see
that under the bijection of Heitsch \cite{Heitsch} $H_n \rightarrow NC(n)$, the
action of $r$ corresponds to reflection across the line bisecting the circle
and going through the vertex labelled $1$.  Therefore, the result follows from
the specialization of Proposition 7.2 to the case $a = 2$ and $b = n$.
\end{proof}

\section{Open Problems}
Most of the results in this paper have enumerated some fixed point set via
some representation theoretic interpretation of the associated action.  These
enumerative results imply a collection of bijections involving standard
tableaux.

\begin{cor}
Let $\lambda \vdash n$ be a rectangular partition and suppose $d | n$.  Then,
there is a bijective correspondence between the set of tableaux in
$SYT(\lambda)$ fixed by $j^d$ and the set of column strict $\frac{n}{d}$-ribbon
tableaux of shape $\lambda$ with content $1^{\frac{n}{d}}$.
\end{cor}

\begin{cor}
Let $\lambda \vdash n$ be a rectangular partition and suppose $d | k$ are
positive integers.  Then, there is a bijective correspondence between the set
of tableaux in $CST(\lambda, k)$ fixed by $j^d$ and the set of column strict
$\frac{k}{d}$-ribbon tableaux of shape $\lambda$ with labels drawn from
$[\frac{k}{d}]$.
\end{cor}

\begin{cor}
Let $\lambda \vdash n$ be a rectangular partition and let $\alpha \models n$ be
a composition with $\ell(\alpha) = k$.  Suppose that $d | k$ and $c_k^d \cdot
\alpha = \alpha$.  Then, there is a bijective correspondence between the set of
tableaux in $CST(\lambda, k, \alpha)$ which are fixed by $j^d$ and the set of
column strict $\frac{k}{d}$-ribbon tableaux of shape $\lambda$ and content
$(\alpha_1, \alpha_2, \dots, \alpha_d)$.
\end{cor}

These combinatorial results naturally lead to the following problem.

\begin{prob}
Write down an explicit bijection between any of the above three pairs of sets.
\end{prob}

As we mentioned, one of the reasons that we required $\lambda$ to be
rectangular in Theorems 1.3, 1.4, and 1.5 was that $j^{n}$ does not in general
fix every element of $SYT(\lambda)$ for arbitrary partitions $\lambda \vdash
n$.  This leads to the following question.

\begin{question}
Let $\lambda \vdash n$ be an arbitrary partition.  What is the order of the
promotion operator $j$ on $SYT(\lambda)$?  What about on $CST(\lambda, k)$ for
some fixed $k \geq 0$?  What about on $CST(\lambda, k, \alpha)$ for some fixed
$k \geq 0$ and composition $\alpha \models n$ with $\ell(\alpha) = k$?
\end{question}

Related to determining the order of promotion on nonrectangular tableaux is the
problem of determining the cycle structure of its action.

\begin{question}
Let $\lambda \vdash n$ be an arbitrary partition and let the cyclic group $C$
act on $SYT(\lambda)$ where the action is generated by jeu-de-taquin promotion.
Find an explicit polynomial $X(q)$ such that $(SYT(\lambda), C, X(q))$ exhibits
the CSP.  Do the same for column strict tableaux, with or without fixed
content.
\end{question}

\section{Acknowledgments}
The author would like to thank Serge Fomin, Vic Reiner, Mark Skandera, Dennis Stanton, John Stembridge, Muge Taskin, and Dennis White for many helpful conversations.  Also, the author would like to thank the anonymous referees for many helpful comments regarding correcting errors, improving exposition, and in particular pointing out a (as presented corrected) flaw in the proof of Lemma 3.3.

\bibliography{../bib/my}
\end{document}